\documentclass[12pt]{article}
\usepackage{amsmath,amssymb}
\usepackage{amsthm}
\usepackage{epsfig}
\usepackage{graphics}
\usepackage{graphicx}
\usepackage{colordvi}
\usepackage{mathrsfs} 
\usepackage{stmaryrd}
\usepackage{geometry}
\usepackage{dsfont}
\usepackage{tikz}
\usetikzlibrary{automata}
\usepackage{url}
\usepackage[colorlinks=off]
{hyperref}
\hypersetup{
    colorlinks=true,
   linkcolor=blue!75!black,
    filecolor=blue!75!black,      
    urlcolor=blue!75!black, 
    citecolor=blue!75!black,
}
\usepackage[textsize=small]{todonotes}
\usepackage[margin=1cm, size=small]{caption}
\usepackage{multirow}
\allowdisplaybreaks[4]
\oddsidemargin
-.5in \evensidemargin -.5in\marginparwidth 0.4in
\usepackage{color,soul}
\usepackage{xcolor}

\usepackage{etoolbox}
\patchcmd{\thebibliography}{\section*}{\section}{}{}
\usepackage{tocloft}

\newcommand{\cc}{{^\circ}}
\newcommand{\1}{\mathds 1}
\newcommand{\G}{\mathscr G}

\newcommand{\F}{\mathscr F}
\newcommand{\B}{\mathscr B}

\newcommand{\ms}{\mathscr}
\renewcommand{\P}{{\mathbb P}}
\newcommand{\E}{{\mathbb E}}

\newcommand{\R}{{\Bbb R}}

\newcommand{\mix}{{\rm mix}}

\newcommand{\defeq}{\stackrel{\rm def}{=}}

\newcommand{\tr}{{\sf tr}}
\renewenvironment{proof}[1][\proofname]{\noindent {\bfseries #1.}\;}{\hfill\ensuremath{\blacksquare}\\}
\renewcommand{\d}{{\mathrm d}}
\newcommand{\e}{{\rm e}}
\newcommand{\ok}{{\overline{\kappa}}}

\newcommand{\eqspace}{{\quad \,}}

\usepackage{currfile}
\usepackage{fancyhdr}
\setstcolor{red}

\newtheoremstyle{slantthm}{10pt}{10pt}{\slshape}{}{\bfseries}{}{.5em}{\thmname{#1}\thmnumber{ #2}\thmnote{ (#3)}.}
\newtheoremstyle{slantrmk}{10pt}{10pt}{\rmfamily}{}{\bfseries}{}{.5em}{\thmname{#1}\thmnumber{ #2}\thmnote{ (#3)}.}

\begin{document}
\theoremstyle{slantthm}
\newtheorem{thm}{Theorem}[section]
\newtheorem{prop}[thm]{Proposition}
\newtheorem{lem}[thm]{Lemma}
\newtheorem{cor}[thm]{Corollary}
\newtheorem{defi}[thm]{Definition}
\newtheorem{disc}[thm]{Discussion}
\newtheorem*{nota}{Notation}
\newtheorem{conj}[thm]{Conjecture}
\newtheorem*{thmm}{Theorem}

\theoremstyle{slantrmk}
\newtheorem{ass}[thm]{Assumption}
\newtheorem{rmk}[thm]{Remark}
\newtheorem{eg}[thm]{Example}
\newtheorem{que}[thm]{Question}
\numberwithin{equation}{section}
\newtheorem{quest}[thm]{Quest}
\newtheorem{prob}[thm]{Problem}

\title{\vspace{-.6cm}\bf The replicator equation in stochastic spatial evolutionary games}
\author{Yu-Ting Chen\footnote{Department of Mathematics and Statistics, University of Victoria, British Columbia, Canada.}\,\,\footnote{Email: \url{chenyuting@uvic.ca}}}
\date{\today
\vspace{-.3cm}
}

\maketitle
\vspace{-.3cm}
\abstract{We study the multi-strategy stochastic evolutionary game with death-birth updating in expanding spatial populations of size $N\to \infty$. The model is a voter model perturbation. For typical populations, we require perturbation strengths satisfying $1/N\ll w\ll 1$. Under appropriate conditions on the space, the limiting density processes of strategy are proven to obey the replicator equation, and the normalized fluctuations converge to a Gaussian process with the Wright--Fisher covariance function in the limiting densities. As an application, we resolve in the positive a conjecture from the biological literature that the expected density processes approximate the replicator equation on many non-regular graphs.\medskip

\smallskip 

}
\vspace{-.2cm}

\tableofcontents

\section{Introduction}\label{sec:intro} 
The replicator equation is the most widely applied dynamics among evolutionary game models. It also gives the first dynamical model in evolutionary game theory \cite{TJ}, thereby establishing an important connection to theoretical explanations of animal behaviors \cite{MP}. For this application, the derivation of the equation considers a large well-mixed population. The individuals implement strategies from a finite set $S$ with $\#S\geq  2$, such that the payoff to $\sigma\in S$ from playing against $\sigma'\in S$ is $\Pi(\sigma,\sigma')$. In the continuum, the density $X_\sigma$ of strategy $\sigma$ evolves with a per capita rate given by the difference between its payoff 
\begin{align}\label{def:Fsigma}
F_{\sigma}(X)=\sum_{\sigma'\in S}\Pi(\sigma,\sigma')X_{\sigma'}
\end{align}
and the average payoff of the population, where $X=(X_{\sigma'};\sigma'\in S)$.
 Hence, the vector density process of strategy  obeys the following replicator equation:
\begin{align}\label{eq:replicator}
\dot{X}_\sigma=X_\sigma\left(F_\sigma(X)-\sum_{\sigma''\in S}F_{\sigma''}(X)X_{\sigma''}\right),\quad \sigma\in S.
\end{align}
This equation is a point of departure for studying connections between the payoff matrix $(\Pi(\sigma,\sigma'))_{\sigma,\sigma'\in S}$, and the equilibrium states of the model by methods from dynamical systems. See \cite{Cressman,HS} for an introduction and more properties. The replicator equation also arises from the Lotka--Volterra equation of ecology and Fisher's fundamental theorem of natural selection \cite{HS,SS}. 

In this paper, we consider the stochastic evolutionary game dynamics in large finite structured populations. Our goal is to prove that the vector density processes of strategy converge to the replicator equation. In this direction, one of the major results in the biological literature is the convergence to the replicator equation on large random regular graphs~\cite{ON}. The authors further conjecture that their approximations extend to more general graphs \cite[Section~5]{ON}. To obtain the proofs, we view the model as a perturbation of the voter model, since this viewpoint has made possible several mathematical results for it (e.g. \cite{CDP:2,C:BC, CD,CMN,ALCFYN,C:EGT}). Our starting point here is the method in \cite{C:EGT}, extended from \cite{CCC,CC}, for proving the diffusion approximations of the time-changed density processes of strategy under \emph{weak selection}. In that context, the corresponding perturbations away from the voter model use strengths typically given by $w=\mathcal O(1/N)$, where $N$ is the population size. 

The questions from \cite{ON} nevertheless concern very different properties. The crucial step of the method in \cite{C:EGT} develops along the equivalence of probability laws in the limit between the evolutionary game model and the voter model. Now, this property breaks down for nontrivial parameters according to the limiting equation from \cite{ON}; distributional limits of the density processes under the evolutionary game and the voter model degenerate to delta distributions of distinct deterministic functions as solutions of differential equations. As will be explained in this introduction, the convergence to the replicator equation also requires the different range of perturbation strengths $w$ satisfying $1/N\ll w\ll 1$. This stronger perturbation implies weaker relations between the two models, and thus, calls for perturbation estimates of the evolutionary game model by  the voter model generalizing those in \cite{C:EGT}. With this change of perturbation strengths, the choice of time changes for the density processes and the characterization of coefficients of the limiting equation are the further tasks to be settled.

Before further explanations of the main results of \cite{ON,C:EGT}, let us specify the evolutionary game model considered throughout this paper. First, to define spatial structure, we impose directed weights $q(x,y)$ on all  pairs of sites $x$ and $y$ in a given population of size $N$. We assume that $q$ is an irreducible probability transition kernel with a zero trace $\sum_{x}q(x,x)=0$. The perturbation strength $w>0$ defines the selection intensity of the model in the following form: for an individual at site $x$ using strategy $\xi(x)$, its interactions with the neighbors determine the fitness as the sum $1-w+w\sum_{y}q(x,y)\Pi\big(\xi(x),\xi(y)\big)$, where $\xi(y)$ denotes the strategy held by the neighbor at $y$. Under the condition of positive fitness by tuning the selection intensity appropriately, the death-birth updating requires that in a transition of state, an individual is chosen to die with rate $1$. Then the neighbors compete to fill in the site by reproduction with probability proportional to the fitness. Although the main results of this paper extend to include mutations of strategies, we relegate this additional mechanism until Section~\ref{sec:mainresults}. 

Besides the case of mean-field populations, the density processes of strategy play a significant role in the biological literature for studying equilibrium states of spatial evolutionary games. Here and throughout this paper, we refer the density of $\sigma$ under population configuration $\xi$ to the weighted sum $\sum_x \pi(x)\1_\sigma\big( \xi(x)\big)$, where $\pi$ is the stationary distribution associated with the transition probability $q$. For such macroscopic descriptions of the model, the critical issue arises from the non-closure of the stochastic equations. The density processes are projections of the whole system, and in general, the density functions are not Markov functions in the sense of \cite{RP}. More specifically, for the evolutionary game with death-birth updating introduced above, the microscopic dynamics from pairwise interactions determine the densities' dynamics. It is neither clear how to reduce the densities' dynamics analytically in the associated Kolmogorov equations. See \cite[Section~1]{C:EGT} for more details on these issues and the physics method discussed below. 

One of the main results in \cite{ON} shows that for selection intensities $w\ll 1$, the \emph{expected} density processes on  large random $k$-regular graphs for any integer $k\geq 3$ approximately obey the following extended form of the replicator equation :
\begin{align}\label{eq:replicatorext}
\dot{X}_\sigma=wX_\sigma\left(F_\sigma(X)+\widetilde{F}_\sigma(X)-\sum_{\sigma''\in S}F_{\sigma''}(X)X_{\sigma''}\right),\quad \sigma\in S.
\end{align}
Here, $F_\sigma(X)$ and $\widetilde{F}_\sigma(X)$ are linear functions in $X=(X_{\sigma'};\sigma'\in S)$ such that the constant coefficients are explicit in the payoff matrix and the graph degree $k$. See \cite[Equations (22) and (36)]{ON}. Note that \eqref{eq:replicatorext} underlines the nontrivial effect of spatial structure, since the coefficients are very different from those for the replicator equation \eqref{def:Fsigma} in mean-field populations. For the derivation of \eqref{eq:replicatorext}, \cite{ON} applies the physics method of pair approximation. It enables the asymptotic closure of the equations of the density processes by certain moment closure approximations and circumvents the fundamental issues discussed above. Moreover, based on computer simulations from \cite{OHLN}, the authors of \cite{ON} conjecture that the approximate replicator equation for the density processes  applies to many non-regular graphs, provided that the constant graph degree $k$ in the coefficients of the replicator equation is replaced by the corresponding average degree. In approaching this conjecture, it is still not clear on how the average degrees of graphs enter. The method in this paper does not extend to this generality either. On the other hand, even within the scope of large random regular graphs, the constant graph degrees and the locally regular tree-like property seem essential in \cite{ON}. We notice that locally tree-like spatial structures are known to be useful to pair approximations in general \cite{SF}. 

In the case of two strategies, the supplementary information (SI) of \cite{OHLN} shows that the density processes of a fixed one approximate the Wright--Fisher diffusion with drift. The derivation also applies pair approximations on large random regular graphs, although it is noticeably different from the derivation in \cite{ON} for the replicator equation on graphs. (A slow-fast dynamical system for the density and a certain rapidly convergent local density is considered in \cite[SI]{OHLN}.) It is neither clear how to justify the derivation in \cite[SI]{OHLN} mathematically. On the other hand, the diffusion approximations of the density processes can be proven on large finite spatial structures subject to appropriate, but general, conditions \cite[Theorem~4.6]{C:EGT} that include random regular graphs as a special case. See \cite{CFM06,CFM08} for mathematical investigations of moment closure in other spatial biological models and some general discussions, among other mathematical works in this direction.

The method in \cite{C:EGT} begins with the aforementioned asymptotic equivalence of probability laws via perturbations for $w=\mathcal O(1/N)$ (not just equivalence of laws of the density processes). The $1/N$-threshold is sharp such that the critical case yields nontrivial drifts of the limiting diffusions. This relation reduces the convergence of the game density processes to a convergence problem of the voter model. For the latter, fast mixing of spatial structure ensures approximations of the coalescence times in the ancestral lineages by analogous exponential random variables from large mean-field populations. This method goes back to \cite{Cox}. Moreover, for the voter density processes, the relevant coalescence times can be reduced to  the  meeting times for two independent copies of the stationary Markov chains over the populations. The almost exponentiality of hitting times \cite{Aldous:AE, AB,AF:MC} applies to these times and leads to the classical diffusivity in the voter density processes in general spatial populations \cite{CCC,CC}. The first moments of these meeting times are also used to time change the densities for the convergence. 

Besides the methods, the convergence results in \cite[Theorem~4.6]{C:EGT} for the game density processes under the specific setting of large random regular graphs and the payoff matrices for prisoner's dilemma games are closely related to the replicator equation on graphs from \cite{ON}. See \eqref{prisoner}  for these payoff matrices and \cite{C:MT} for the exact asymptotics $N(k-1)/[2(k-2)]$,  $N\to\infty$, of the expected meeting times on the large random $k$-regular graphs. In this case, the diffusion approximations in \cite[SI]{OHLN} hold to the degree of matching constants if the time changes are formally undone \cite{C:MT}. (See also \cite[Remark~3.1]{C:MT} for a correction of inaccuracy in \cite{C:EGT} on passing limits along random regular graphs.) This standpoint extends to a recovery of the replicator equation on graphs from \cite{ON} by a similar formal argument. It shows that these results in \cite{ON,OHLN}, both due to pair approximations, are algebraically consistent with each other. See the end of Section~\ref{sec:mainresults} for details and the second main result discussed below for further comparison. In addition to its own interest, the replicator equation on graphs concerns a unified characterization of the evolutionary game within an enlarged range of selection intensities  as mentioned above. 

The main results of this paper obtain the convergence to the replicator equation under the above specific setting, in addition to extensions to general spatial populations and payoff matrices. Multiple strategies and mutations are allowed. See Theorem~\ref{thm:main} and Corollary~\ref{cor:symmetric}. For the extended context, the first  main result [Theorem~\ref{thm:main} (1$\cc$)] proves the convergence of the vector density processes of strategy under the following assumptions. We require that the stationary distributions associated with the spatial structures are asymptotically comparable to the uniform distributions (see \eqref{cond:pi}), and these spatial structures allow for suitable time changes of the density processes and suitable selection intensities (Definition~\ref{def:admissible}). Here, for the typical eligible populations, the time changes can range in $1\ll\theta \ll N$. The selection intensities are \emph{of the inverse order} so that $1/N\ll w\ll1 $. Then the precise limiting equation is given by \eqref{eq:replicatorext}, with the selection intensity $w$ replaced by a constant $w_\infty$ as a limit of the parameters. The proof also determines $F_\sigma(X)$ and $\widetilde{F}_\sigma(X)$ for \eqref{eq:replicatorext}:
\begin{align}
F_\sigma(X)&=\overline{\kappa}_{0|2|3}\sum_{\sigma'\in S}\Pi(\sigma,\sigma')X_{\sigma'},\label{F1}\\
\begin{split}
\widetilde{F}_\sigma(X)&=(\overline{\kappa}_{(2,3)|0}-\overline{\kappa}_{0|2|3})\Pi(\sigma,\sigma)\\
&\quad +\sum_{\sigma'\in S}(\ok_{(0,3)|2}-\ok_{0|2|3})[\Pi(\sigma,\sigma')-\Pi(\sigma',\sigma)]X_{\sigma'}\\
&\quad -(\overline{\kappa}_{(2,3)|0}-\overline{\kappa}_{0|2|3})\sum_{\sigma'\in S}\Pi(\sigma',\sigma')X_{\sigma'}.\label{F2}
\end{split}
\end{align}
Here, $\overline{\kappa}_{(2,3)|0}$, $\ok_{(0,3)|2}$, and $\overline{\kappa}_{0|2|3}$ are nonnegative constants defined by the asymptotics of some coalescent characteristics of the spatial structures. See  Section~\ref{sec:slow} for the definitions of these constants. 

The first main result [Theorem~\ref{thm:main} (1$\cc$)] has the meaning of spatial universality as the diffusion approximations of the voter model and the evolutionary game in \cite{CCC,CC,C:EGT}, although it does not recover the explicit equations obtained in \cite{ON} on large random regular graphs under general payoff matrices. The conditions do not require convergence of local geometry as in the large discrete tori and large random regular graphs. The spatial structures can remain sparse in the limit, which is in stark contrast to the usual assumptions for proving scaling limits of particle systems. The locally tree-like property usually assumed in pair approximations is not required either. Based on these properties, the first main result  [Theorem~\ref{thm:main} (1$\cc$)] gives an answer in the positive for the conjecture in \cite{ON} to the degree of using constants that may depend implicitly on the space: The approximations of the expected density processes by the replicator equation extend to many non-regular graphs, whenever the initial conditions converge deterministically. 
 
To further the formal comparison mentioned above with the approximate Wright--Fisher diffusion from \cite[SI]{OHLN}, the second main result [Theorem~\ref{thm:main} (2$\cc$)] considers one additional aspect for the convergence of the density processes. In this part, the normalized fluctuations are proven to converge to a vector centered Gaussian martingale [Theorem~\ref{thm:main} (2$\cc$)]. The quadratic covariation is the Wright--Fisher diffusion matrix in the limiting densities $X$: $\int_0^{ t} X_\sigma(s)[\delta_{\sigma,\sigma'}-X_{\sigma'}(s)]\d s$, $\sigma,\sigma'\in S$, where $\delta_{\sigma,\sigma'}$ are the Kronecker deltas. For the case of only two strategies on large random regular graphs, this covariation formally recovers the approximate Wright--Fisher diffusion term from \cite[SI]{OHLN}. Note that this result and the convergence to the replicator equation do not imply the diffusion approximations of the density processes. 

In the rest of this introduction, we explain the proof of the first main result. Its investigation raises all the central technical issues pointed out above. First, the lack of an asymptotic equivalence of probability laws is resolved via the populations' microscopic dynamics driving the density processes. Duhamel's principle replaces the pathwise, global change of measure method in \cite{C:EGT} and shows the irrelevance of selection intensities in the microscopic dynamics (Proposition~\ref{prop:duhamel}). This approach then links  to the decorrelation proven in \cite[Section~4]{CCC} for some ``local'' meeting time distributions, from the ancestral lineages, driving the dynamics of the voter density processes. Here, local meeting times refer to those where the initial conditions of the Markov chains are within fixed numbers of edges. 

The decorrelation property from \cite{CCC} shows that the probability distributions of those particular local meeting times for general populations converge to nontrivial convex combinations of the delta distributions at zero and infinity. The time scales are slower than those for the diffusion approximations.  In particular, the exponential distribution is not just absent. No distributions with a nonzero mass between zero and infinity arise in the limit. Informally speaking, the decorrelation occurs at time scales between the period when details of the spatial structures dominate and the period when the almost exponentiality \cite{Aldous:AE, AB,AF:MC} plays a role. To us, this presence of multiple time scales in the evolutionary dynamics is reminiscent of the slow-fast dynamical system in \cite[SI]{OHLN}.

For the convergence to the replicator equation, the choice of the time changes for the densities and the characterization of the limiting equation use the decorrelation from \cite{CCC} and its extensions. First, the time changes can only grow slower than those for the diffusion approximations since the limiting trajectories are less rougher.  This requirement relates the convergence to the decorrelation. We are now interested in proving the best possible range of growing time changes for the decorrelation, not just using the particular ones from \cite{CCC}. After all, in \cite{ON}, the replicator equation is expected to be present within the broad range $w\ll 1$, and our argument requires the selection intensities to be of the inverse order of the time changes. Moreover, the application of Duhamel's principle mentioned above leads to the entrance of various local meeting times more than those for the voter densities in \cite{CCC}. Simultaneous decorrelation in these local meeting times is essential for getting a deterministic limiting differential equation: This property involves asymptotic path regularity of the density processes. The constant coefficients in \eqref{F1} and \eqref{F2} also arise as the weights at infinity in the limiting local meeting time distributions for the typical eligible populations. See Sections~\ref{sec:slow} and \ref{sec:eqn} for the related proofs.  \medskip   

\noindent {\bf Organization.} Section~\ref{sec:mainresults} introduces the evolutionary game model and the voter model analytically and discusses the main results (Theorem~\ref{thm:main} and Corollary~\ref{cor:symmetric}). In Section~\ref{sec:dynamics}, we define the voter model and the evolutionary game model as semimartingales and briefly explain the role of the coalescing duality. In Section~\ref{sec:slow}, we quantify the time changes in proving the main results and characterize the coefficients of the limiting equation. Section~\ref{sec:eqn} is devoted to the main arguments of the proofs of Theorem~\ref{thm:main} and Corollary~\ref{cor:symmetric}. Finally, Section~\ref{sec:coal} presents some auxiliary results for coalescing Markov chains.\medskip

\noindent {\bf Acknowledgments}
The author would like to thank Lea Popovic for comments on earlier drafts and Sabin Lessard for pointing out several references from the literature. Support from the Simons Foundation before the author's present position and from the Natural Science and Engineering Research Council of Canada is gratefully acknowledged.

\section{Main results}\label{sec:mainresults}
In this section, we introduce the stochastic spatial evolutionary game with death-birth updating in more detail. A discussion of the main results of this paper then follows. To be consistent with the viewpoint of voter model perturbations and the neutral role of the voter model, strategies will be called {\bf types} in the rest of this paper. The settings here and in the next section are adapted from those in \cite{CCC,CC,C:EGT} to the context of evolutionary games with multiple types. 

Recall that a discrete spatial structure considered in this paper is given by an irreducible, reversible probability kernel $q$ on a finite nonempty set $E$ such that $\tr(q)=\sum_{x\in E}q(x,x)=0$. Write $N=\#E$ and $\pi$ for the unique stationary distribution of $q$. The interactions of individuals are  defined by a payoff matrix $\Pi=(\Pi(\sigma,\sigma'))_{\sigma,\sigma'\in S}$ of real entries. Fix $\overline{w}\in (0,\infty)$ such that
\begin{align}\label{def:wbar}
w+w\sum_{z\in E}q(y,z)|\Pi\big(\xi(y),\xi(z)\big)|<1,\quad \forall\;w\in [0,\overline{w}],\;y\in E.
\end{align}
Then the following perturbed transition probability is used to update types of individuals due to interactions:
\begin{align}
q^w(x,y,\xi)&\defeq  \frac{q(x,y)\left[(1-w)+w\sum_{z\in E}q(y,z)\Pi\big(\xi(y),\xi(z)\big)\right]}{\sum_{y'\in E}q(x,y')\left[(1-w)+w\sum_{z\in E}q(y',z)\Pi\big(\xi(y'),\xi(z)\big)\right]}.\label{def:qw}
\end{align}
With these updates and the updates based on a mutation measure $\mu$ on $S$, two types of configurations $\xi^{x,y},\xi^{x|\sigma}\in S^E$ result. They are obtained from $\xi\in S^E$ by changing only the type at $x$ such that  $\xi^{x,y}(x)=\xi(y)$ and $\xi^{x|\sigma}(x)=\sigma$. Hence, the evolutionary game $(\xi_t)$ is a Markov jump process with a generator given by
\begin{align}\label{def:Lw}
\begin{split}
\mathsf L^{w} H(\xi)=&\sum_{x,y\in E}q^w(x,y,\xi)[H(\xi^{x,y})-H(\xi)]\\
&+\sum_{x\in E}\int_S [H(\xi^{x|\sigma})-H(\xi)]\d \mu(\sigma),\quad H:S\to \R.
\end{split}
\end{align}
The first sum on the right-hand side of \eqref{def:Lw} governs changes of types due to selection, and the second sum is responsible for mutations. Given $\xi\in S^E$ and a probability distribution $\nu$ on $S^E$ as initial conditions, we write $\P^w_\xi$ and $\E^w_\xi$, or $\P^w_\nu$ and $\E^w_\nu$, under the laws associated  with $\mathsf L^w$. For $w=0$, the generator $\mathsf L^{w}$ is reduced to the generator  $\mathsf L$ of the multi-type voter model with mutation, and the notation $\P$ and $\E$ is used.

The object in this paper is the vector density processes $p(\xi_t)=(p_\sigma(\xi_t);\sigma\in S)$ for the evolutionary game with death-birth updating. Here, the density function of $\sigma\in S$ is given by
\begin{align}\label{def:density}
p_\sigma(\xi)=\sum_{x\in E}\1_\sigma\circ\xi(x)\pi(x),
\end{align}
where $f\circ \xi(x)=f(\xi(x))$. Under $\P^w$, $p_\sigma(\xi_t)$ admits a semimartingale decomosition:
\begin{align}\label{density:dynamics}
p_\sigma(\xi_t)=p_\sigma(\xi_0)+A_\sigma(t)+M_\sigma(t),
\end{align}
where $A_\sigma(t)=\int_0^t \mathsf L^w p_\sigma(\xi_s)\d s $. In the sequel, we study the convergence of the vector density processes and the martingales $M_\sigma$ separately, along an appropriate sequence of discrete spatial structures $(E_n,q^{(n)})$ with $N_n=\#E_n\to\infty$. \medskip

\noindent {\bf Convention for superscripts and subscripts.} Objects associated with $(E_n,q^{(n)})$ will carry either superscripts ``$(n)$'' or subscripts ``$n$'', although additional properties may be assumed so that these objects are not just based on $(E_n,q^{(n)})$. Otherwise, we refer to a fixed spatial structure $(E,q)$. \hfill $\blacksquare$\medskip

For the main theorem, we choose parameters as time changes for the density processes, mutation measures, and selection intensities. The choice is according to the underlying discrete spatial structures. We use $\nu_n(\1)=\sum_{x\in E_n}\pi^{(n)}(x)^2$ and the first moment $\gamma_n$ of the first meeting time of two independent stationary rate-$1$ $q^{(n)}$-Markov chains. The other characteristic of the spatial structure considers the mixing time $\mathbf t^{(n)}_{\rm mix}$ of the $q^{(n)}$-Markov chains  and the spectral gap $\mathbf g_n$ as follows. Recall that the semigroup of the continuous-time rate-$1$ $q^{(n)}$-Markov chain is given by $(\e^{t(q^{(n)}-1)};t\geq 0)$. With
\begin{align}\label{def:dE}
d_{E_n}(t)=\max_{x\in E_n}\big\|\e^{t(q^{(n)}-1)}(x,\cdot)-\pi^{(n)}\big\|_{\rm TV}
\end{align}
for $\|\cdot\|_{\rm TV}$ denoting the total variation distance, we choose
\begin{align}
\label{def:tmix}
\mathbf t^{(n)}_{\rm mix}=\inf\{t\geq 0;d_{E_n}(t)\leq (2\e)^{-1}\}.
\end{align}
The spectral gap $\mathbf g_n$ is the distance between the largest and second largest eigenvalues of $q^{(n)}$.

\begin{defi}\label{def:admissible}
For all $n\geq 1$,  let $\theta_n\in (0,\infty)$ be a time change,  $\mu_n$  a mutation measure on $S$, and $w_n\in [0,\overline{w}]$.  The sequence $(\theta_n,\mu_n,w_n)$ is said to be {\bf admissible} if all of the following conditions hold. First, $(\theta_n)$ satisfies
\begin{align}\label{cond1:thetan}
\lim_{n\to\infty}\theta_n=\infty,\quad \lim_{n\to\infty}\frac{\theta_n}{\gamma_n}<\infty,\quad 
\lim_{n\to\infty}\gamma_n\nu_n(\1)\e^{-t\theta_n}= 0,\quad \forall\;t\in (0,\infty),
\end{align}
and at least one of the two mixing conditions holds:  
\begin{align}\label{cond2:thetan}
\lim_{n\to\infty}\gamma_n\nu_n(\1)\e^{-\mathbf g_n \theta_n}=0\quad\mbox{or}\quad \lim_{n\to\infty}\frac{\mathbf t^{(n)}_{\rm mix}}{\theta_n}[1+\log^+(\gamma_n/\mathbf t^{(n)}_{\rm mix})]=0,
\end{align}
where $\log^+\alpha=\log (\max\{\alpha,1\})$. Second, we require the following limits for $(\mu_n)$ and $(w_n)$: 
\begin{align}
& \lim_{n\to\infty} \mu_n(\sigma)\theta_n=\mu_\infty(\sigma)<\infty,\quad \forall\;\sigma\in S;\label{def:mun}\\
& \lim_{n\to\infty}\ w_n=0,\quad 
\lim_{n\to\infty}\frac{w_n\theta_n }{2\gamma_n\nu_n(\1)}=w_\infty<\infty,\quad \limsup_{n\to\infty}w_n\theta_n<\infty.
\label{def:wn}
\end{align}
\end{defi}

Another condition of the main theorem requires that $\sup_nN_n\max_{x\in E_n}\pi^{(n)}(x)<\infty$, which implies $\gamma_n\geq \mathcal O(N_n)$ (see \eqref{ergodic} or \cite[(3.21)]{CCC} for details). In this context, the admissible $\theta_n$ has the following effects. If $\lim_{n}\theta_n/\gamma_n\in (0,\infty)$, the time-changed density processes $p_1(\xi_{\theta_nt})$  of the voter model converge to the Wright--Fisher diffusion \cite{CCC,CC}. Moreover, the density processes of the evolutionary game converge to the same diffusion but with a drift \cite[Theorem~4.6]{C:EGT}. These diffusion approximations hold under the mixing conditions slightly different  from those in \eqref{cond2:thetan}. Therefore, assuming $\lim_{n}\theta_n/\gamma_n=0$ in \eqref{cond1:sn} has the heuristic that the time-changed density processes have paths  less rougher in the limit, and so, do not converge to diffusion processes. Note that this variation of time scales can be contrasted with, e.g., the context considered in \cite{EN:80} where, among other results,  the discrete processes converge to the equilibrium states of the limiting process due to faster time changes.

The other conditions for the admissible sequences mainly consider the typical case of ``transient'' spatial structures. The kernels are characterized by the condition $\sup_n \gamma_n\nu_n(\1)<\infty$ \cite[Remark~2.4]{CCC}. In this case, \eqref{cond1:thetan} can be satisfied by any sequence $(\theta_n)$ such that $1\ll \theta_n\ll N_n$, and \eqref{def:mun} and \eqref{def:wn} allow for nonzero $\mu_\infty$ and $w_\infty$. The somewhat tedious condition in \eqref{def:wn} simplifies drastically, and we get $N_n^{-1}\ll w_n\ll 1$ when $\lim_n w_n\theta_n$ is nonzero. As for the mixing conditions in \eqref{cond2:thetan}, they can pose severe limitations if the spatial structures are ``recurrent'' ($\sup_n \gamma_n\nu_n(\1)=\infty$). In this case, we may not be able to find admissible sequences such that $w_\infty>0$, so that the limiting equation to be presented below only allows constant solutions in the absence of mutation.  
For example, the two-dimensional discrete tori satisfy $\gamma_n\nu_n(\1)\sim C\log N_n$, $\mathbf t^{(n)}_{\rm mix}\leq \mathcal O(N_n)$ and $\mathbf g_n=\mathcal O(1/N_n)$. See  \cite{Cox} and \cite[Theorem~10.13 on p.133, Theorem~5.5 on p.66 and Section~12.3.1 on p.157]{LPW}.  We can choose $\theta_n=N_n(\log\log N_n)^2$ to satisfy \eqref{cond1:thetan} with $\lim_n\theta_n/\gamma_n=0$, and the first mixing condition in \eqref{cond2:thetan}. But now the admissible $(w_n)$ only gives $w_\infty=0$. We notice that a similar restriction is pointed out in \cite{Cox:Feller} on the low density scaling limits of the biased voter model, where the limit is Feller's branching diffusion with drift. 

From now on, we write  $\pi_{\min}=\min_{x\in E}\pi(x)$ and  $\pi_{\max}=\max_{x\in E}\pi(x)$ for the stationary distribution $\pi$ of $(E,q)$. The main theorem stated below shows a law of large numbers type convergence for the density processes and a central limit theorem type convergence for the fluctuations. These two results do not combine to give the diffusion approximation of the density processes proven in \cite{C:EGT}.

\begin{thm}\label{thm:main}
Let $(E_n,q^{(n)})$ be a sequence of irreducible, reversible probability kernels defined on finite sets with $N_n=\#E_n\to\infty$. Assume the following conditions:
\begin{enumerate}
\item [\rm (a)] Let $\nu_n$ be a probability measure on $S^{E_n}$ such that $\nu_n(\xi;p(\xi)\in \cdot)$ converges in distribution to a probability measure $\overline{\nu}_\infty$ on $[0,1]^S$.
\item [\rm (b)] It holds that 
\begin{align}\label{cond:pi}
0<\liminf_{n\to\infty}N_n\pi^{(n)}_{\min}\leq \limsup_{n\to\infty}N_n\pi^{(n)}_{\max}<\infty.
\end{align}
\item [\rm (c)] The limits in \eqref{def:|kell1|ell2} and \eqref{def:|||} defining the nonnegative constants $\overline{\kappa}_{(2,3)|0}$, $\overline{\kappa}_{(0,3)|2}$ and $\ok_{0|2|3}$ exist. These constants depend only on space. 
\item [\rm (d)] We can choose an admissible sequence $(\theta_n,\mu_n,w_n)$ as in Definition~\ref{def:admissible} such that $\lim_n\theta_n/\gamma_n=0$.
\end{enumerate}
Then the following convergence in distribution of processes holds:
\begin{enumerate}
\item [\rm (1$\cc$)] The sequence of the vector density processes $\big(p(\xi_{\theta_nt}),\P^{w_n}_{\nu_n}\big)$ converges to the solution $X$ of the following differential  equation with the random initial condition $\P(X_0\in \cdot)=\overline{\nu}_\infty$:
\begin{align}
\begin{split}\label{p1:lim}
\dot{X}_\sigma&=w_\infty X_\sigma\left(F_\sigma(X)+\widetilde{F}_\sigma(X)-\sum_{\sigma''\in S}F_{\sigma''}(X)X_{\sigma''}\right)\\
&\quad\; +\mu_\infty(\sigma)(1-X_\sigma)-\mu_\infty(S\setminus\{\sigma\}) X_\sigma ,\quad \sigma\in S,
\end{split}
\end{align}
where $F_\sigma(X)$ and $\widetilde{F}_\sigma(X)$ are linear functions in $X$ defined by \eqref{F1} and \eqref{F2}. 
Moreover, the sum of the $\ok$-constants in $F_\sigma(X)$ and $\widetilde{F}_\sigma(X)$ is nontrivial to the following degree: 
\begin{align}\label{kappa:>}
(\overline{\kappa}_{(2,3)|0}-\overline{\kappa}_{0|2|3})+ \overline{\kappa}_{0|2|3}+(\ok_{(0,3)|2}-\ok_{0|2|3})\in (0,\infty).
\end{align}

\item [\rm (2$\cc$)] Recall the vector martingale defined by \eqref{density:dynamics}, and set $M^{(n)}_\sigma(t)=(M_\sigma(\theta_nt);\sigma\in S)$ under $\P^{w_n}_{\nu_n}$. If, moreover, $\lim_n \gamma_n\nu_n(\1)/\theta_n=0$ holds, then $(\gamma_n/\theta_n)^{1/2}M^{(n)}$ converges to a vector centered Gaussian martingale with quadratic covariation $(\int_0^tX_\sigma(s)[\delta_{\sigma,\sigma'}-X_{\sigma'}(s)]\d s;\sigma,\sigma'\in S) $. 
\end{enumerate}
\end{thm}

We present the proof of Theorem~\ref{thm:main} in Section~\ref{sec:eqn}. The existence of the limits in condition (c) is proven in Proposition~\ref{prop:kell}. See Lemma~\ref{lem:tight} for the additional condition in  Theorem~\ref{thm:main} (2$\cc$).

To illustrate Theorem~\ref{thm:main}, we consider the generalized  prisoner's dilemma matrix in the rest of this section. The matrix is for games among individuals of two types : 
\begin{align}\label{prisoner}
\Pi=
\bordermatrix{~ & 1& 0 \cr
              1 & b-c & -c \cr
              0 & b & 0 \cr}
\end{align}
for real entries $b,c$. (The usual prisoner's dilemma matrix requires $b>c>0$.) The proof of the following corollary also appears in Section~\ref{sec:eqn}.

\begin{cor}\label{cor:symmetric}
Let conditions {\rm (a)--(d)} of Theorem~\ref{thm:main} be in force and $\Pi$ be given by \eqref{prisoner}. If, moreover, $q^{(n)}$ are symmetric ($q^{(n)}(x,y)\equiv q^{(n)}(y,x)$) and 
\begin{align}\label{cond:q2}
\lim_{n\to\infty}\gamma_n\nu_n(\1)\pi^{(n)}\{x\in E_n;q^{(n),2}(x,x)\neq q^{(\infty),2}\}=0
\end{align}
for some constant $q^{(\infty),2}$, then the differential equation for $X_1=1-X_0$ takes a simpler form:
\begin{align}\label{eq:X1}
\dot{X}_1=w_\infty(bq^{(\infty),2}-c)X_1(1-X_1)+\mu_\infty(1)(1-X_1)-\mu_\infty(0) X_1.
\end{align}
\end{cor}
\medskip

Corollary~\ref{cor:symmetric} applies to large random $k$-regular graph for a fixed integer $k\geq 3$, with $q^{(\infty),2}=1/k$ and $\gamma_n/N_n\to (k-1)/[2(k-2)]$ (see \eqref{MUU:RG} and the discussion there). Additionally, $(\theta_n)$ can be chosen to be any sequence such that $1\ll \theta_n\ll N_n$, and $(w_n)$ can be any such that $(w_n\theta_n)$ converges in $[0,\infty)$. See \cite{C:MT} and Section~\ref{sec:rrg}. (More precisely, the application needs to pass limits along subsequences, since these graphs are randomly chosen.) Assume the absence of mutation. Then in this case, one can \emph{formally} recover the replicator equation \eqref{eq:X1} from the drift term of the approximate Wright--Fisher diffusion in \cite[SI]{OHLN} as follows. For the density process $p_1(\xi_t)$ under $\P^{w_n}$, that drift term reads
\begin{align}\label{drift}
w_n\cdot \frac{(k-2) (b-ck)}{k(k-1)}p_1(\xi_t)[1-p_1(\xi_t)].
\end{align}
Note that $\gamma_n\approx N_n(k-1)/[2(k-2)]$ as mentioned  above and the choice in \eqref{def:wn} of $w_n$ gives $w_n \approx w_\infty 2\gamma_nN_n^{-1}/\theta_n$. By using these approximations and multiplying the foregoing drift term by $\theta_n$ as a time change,  we get the approximate drift $w_\infty (b/k-c)p_1(\xi_{\theta_nt})[1-p_1(\xi_{\theta_nt})]$ of $p_1(\xi_{\theta_nt})$. This approximation recovers \eqref{eq:X1}. The same formal argument can be used to recover the noise coefficient in Theorem~\ref{thm:main} (2$\cc$). See also \cite[Remark~4.10]{C:EGT} for the case of diffusion approximations.

\section{Semimartingale dynamics}\label{sec:dynamics}
In this section, we define the voter model and the evolutionary game model as solutions to stochastic integral equations driven by point processes. Then we view these equations in terms of semimartingales and identify some leading order terms for  the forthcoming perturbation argument. We recall the coalescing duality for the voter model briefly at the end of this section.

First, given a triplet $(E,q,\mu)$, an equivalent characterization of the corresponding voter model is given as follows. Introduce independent $(\F_t)$-Poisson processes $\{\Lambda(x,y);x,y\in E\}$ and $\{\Lambda^\sigma_t(x);\sigma\in S,x\in E\}$ such that
\begin{align}
\begin{split}\label{rates}
\Lambda_t(x,y)& \quad\mbox{with rate}\quad  \E[\Lambda_1(x,y)]=q(x,y)\quad\mbox{and}\\
\Lambda^\sigma_t(x)&\quad \mbox{with rate}\quad \E[\Lambda^\sigma_1(x)]=\mu(\sigma),\quad x,y\in E,\;\sigma\in S.
 \end{split}
\end{align}
These jump processes are defined on a complete filtered probability space $\big(\Omega,\F,(\F_t),\P\big)$. Then given an initial condition $\xi_0\in S^E$, the $(E,q,\mu)$-voter model can be defined as the pathwise unique $S^E$-valued solution of the following stochastic integral equations \cite{CDP,MT}: for $x\in E$ and $\sigma\in S$, 
\begin{align}
\begin{split}
\1_\sigma\circ \xi_t(x)&=\1_\sigma\circ \xi_0(x)+\sum_{y\in E}\int_0^t [\1_\sigma\circ \xi_{s-}(y)-\1_\sigma\circ \xi_{s-}(x)]\d \Lambda_s(x,y)\\
&\eqspace+\int_0^t \1_{\sigma_{S\setminus\{\sigma\}}}\circ \xi_{s-}(x)\d \Lambda^\sigma_s(x)-\sum_{\sigma'\in S\setminus\{\sigma\}}\int_0^t\1_\sigma\circ \xi_{s-}(x)\d \Lambda^{\sigma'}_s(x).
\label{eq:voter}
\end{split}
\end{align}
Hence, the type at $x$ is replaced and changed to the type at $y$
when $\Lambda(x,y)$ jumps, 
and the type seen at $x$ is $\sigma$ right after $\Lambda^\sigma(x)$ jumps.

Recall that the rates of the evolutionary game are defined by \eqref{def:qw}. With the choice of $\overline{w}$ from \eqref{def:wbar}, $q^w(x,y,\xi)>0$ if and only if $q(x,y)>0$. Hence, Girsanov's theorem for point processes \cite[Section~III.3]{JS} can be applied to change the intensities of the Poisson processes $\Lambda(x,y)$ to $q^w(x,y,\xi)$ such that under a probability measure $\P^w$ equivalent to $\P$ on $\F_t$ for all $t\geq 0$, 
\begin{align}\label{mg:hat}
\widehat{\Lambda}_t(x,y)\defeq\Lambda_t(x,y)-\int_0^t q^w(x,y,\xi_s)\d s\quad\&\quad \widehat{\Lambda}_t^\sigma(x)\defeq\Lambda^\sigma_t(x)-\mu(\sigma)t
\end{align}
are $(\F_t,\P^w)$-martingales. See \cite[Section~2]{C:EGT} for the explicit form of $D^w$ when $S=\{0,1\}$. Since all of $\widehat{\Lambda}(x,y)$ and $\widehat{\Lambda}^\sigma(x)$ do not jump simultaneously under $\P^w$ by the absolute continuity with respect to $\P$, the product of any distinct two of them has a zero predictable quadratic variation \cite[Theorem~4.2, Proposition~4.50, and Theorem~4.52 in Chapter~I]{JS}.

The point processes defined above now allows for straightforward representations of the dynamics of the density processes. By (\ref{eq:voter}), 
\begin{align}
\begin{split}
p_\sigma(\xi_t)
&=p_\sigma(\xi_0)+\sum_{x,y\in E}\pi(x)\int_0^t \big[\1_{\sigma}\circ\xi_{s-}(y)-\1_\sigma\circ\xi_{s-}(x)\big]\d\Lambda_s(x,y)\\
&\quad +
\sum_{x\in E}\pi(x)\int_0^t \1_{S\setminus\{\sigma\}}\circ\xi_{s-}(x)\d \Lambda^\sigma_s(x)\\
&\quad -\sum_{\sigma'\in S\setminus\{\sigma\}}\sum_{x\in E}\pi(x)\int_0^t \1_{\sigma}\circ\xi_{s-}(x)\d \Lambda^{\sigma'}_s(x).\label{dynamics:p1}
\end{split}
\end{align}
To obtain the limiting semimartingale for the density processes, we use the foregoing equation to derive the explicit semimartingale decompositions of the density processes. 

To obtain these explicit decompositions, first, note that the dynamics of $p_\sigma(\xi_t)$ under $\P^w$ relies on various kinds of frequencies and densities as follows. For all $x\in E,\xi\in S^E$ and $\sigma,\sigma_1,\sigma_2\in S$, we set
\begin{align}\label{def:Well}
\begin{split}
f_\sigma(x,\xi)&=\sum_{y\in E}q(x,y)\1_{\sigma}\circ \xi(y),\\
 f_{\sigma_1\sigma_2}(x,\xi)&= \sum_{y\in E}q(x,y)\1_{\sigma_1}(y)\sum_{z\in E}q(y,z)\1_{\sigma_2}\circ\xi(z),\\
f_{\bullet\sigma}(x,\xi)&= \sum_{y\in E}q(x,y)\sum_{z\in E}q(y,z)\1_{\sigma}\circ \xi(z),\quad 
\overline{f}(\xi)=\sum_{x\in E}\pi(x)f(x,\xi).
\end{split}
\end{align} 
To minimize the use of the summation notation, we also express these functions in terms of stationary \emph{discrete-time} $q$-Markov chains $\{U_\ell;\ell\in \Bbb Z_+\}$ and $\{U'_\ell;\ell\in \Bbb Z_+\}$ with $U_0=U_0'$ such that conditioned on $U_0$, the two chains are independent. Additionally, let $(U,U')\sim \pi\otimes \pi$ and $(V,V')$ be distributed as
\begin{align}\label{def:VV'}
 \P(V=x,V'=y)=\frac{\nu(x,y)}{\nu(\1)},\quad x,y\in E,
\end{align}
for $\nu(x,y)=\pi(x)^2q(x,y)$ and $\nu(\1)=\sum_{x,y}\nu(x,y)=\sum_x \pi(x)^2$. (When $q$ is symmetric, $\nu(\1)$ reduces to $N^{-1}$.) For example, $\overline{f_{\sigma_1}f_{\sigma_2\sigma_3}}=\E[\1_{\sigma_1}\circ \xi(U_1')\1_{\sigma_2}\circ \xi(U_1)\1_{\sigma_3}\circ \xi(U_2)]$. We also set
\begin{align}\label{def:p10}
p_{\sigma\sigma'}(\xi)=\E[\1_\sigma\circ \xi(V)\1_{\sigma'}\circ \xi(V')].
\end{align}

Second, we turn to algebraic identities that determine the leading order terms for the forthcoming perturbation arguments.  For $w\in [0,\overline{w}]$, the kernel $q^w$ defined by \eqref{def:qw} can be expanded to the second order in $w$ as follows:
\begin{align}
q^w(x,y,\xi)&=q(x,y)\frac{1-wB(y,\xi)}{1-wA(x,\xi)}\notag\\
&=q(x,y)+\sum_{i=1}^\infty w^iq(x,y)[A(x,\xi)-B(y,\xi)]A(x,\xi)^{i-1}\notag\\
&=
q(x,y)+wq(x,y)[A(x,\xi)-B(y,\xi)]+w^2q(x,y)R^w(x,y,\xi),\label{qwq:exp}
\end{align}
where 
\begin{align*}
A(x,\xi)&=1-\sum_{z\in E}q(x,z)\sum_{z'\in E}q(z,z')\Pi\big(\xi(z),\xi(z')\big),\\
B(y,\xi)&=1-\sum_{z\in E}q(y,z)\Pi\big(\xi(y),\xi(z)\big),
\end{align*}
and $R^w$ is uniform bounded in $w\in [0,\overline{w}],x,y,\xi,(E,q)$.

\begin{lem}\label{lem:D}
For all $\xi\in S^E$ and $\sigma\in S$,
\begin{align}
\overline{D}_\sigma(\xi)&\!\defeq \sum_{x,y\in E}\pi(x)\big[\1_\sigma\circ\xi(y)-\1_\sigma\circ\xi(x)\big]q(x,y)[A(x,\xi)-B(y,\xi)]\label{eq:Dsigma0}\\
&=\sum_{\stackrel{\scriptstyle \sigma_0,\sigma_3\in S}{ \sigma_0\neq\sigma}}\Pi(\sigma,\sigma_3)\overline{f_{\sigma_0}f_{\sigma\sigma_3}}(\xi)-\sum_{\stackrel{\scriptstyle \sigma_2,\sigma_3\in S}{\sigma_2\neq\sigma}}\Pi(\sigma_2,\sigma_3)\overline{f_{\sigma}f_{\sigma_2\sigma_3}}(\xi).\label{eq:Dsigma}
\end{align}
In particular, if $\Pi$ is given by \eqref{prisoner},
then 
\begin{align}\label{eq:Dsigma1}
\overline{D}_1(\xi)=b\overline{f_{1}f_{ \bullet 0}}(\xi)-b\overline{f_{10}}(\xi)-c\overline{f_1f_{0}}(\xi).
\end{align}
\end{lem}
\begin{proof}
By using the reversibility of $q$ and taking $y$ in \eqref{eq:Dsigma0} as the state of $U_0$ in the sequence $\{U_\ell\}$ defined above, we can compute $\overline{D}_\sigma$ as
\begin{align}
\begin{split}\label{Dbar:cal}
\overline{D}_\sigma(\xi)
&=-\sum_{x,y\in E}\pi(x)\1_\sigma\circ\xi(y)q(x,y)\sum_{z\in E}q(x,z)\sum_{z'\in E}q(z,z')\Pi\big(\xi(z),\xi(z')\big)\\
&\eqspace+\sum_{x,y\in E}\pi(x)\1_\sigma\circ\xi(y)q(x,y)\sum_{z\in E}q(y,z)\Pi\big(\xi(y),\xi(z)\big)\\
&\eqspace+\sum_{x,y\in E}\pi(x)\1_\sigma\circ\xi(x)q(x,y)\sum_{z\in E}q(y,z)\Pi\big(\xi(y),\xi(z)\big)\\
&\eqspace-\sum_{x,y\in E}\pi(x)\1_\sigma\circ\xi(x)q(x,y)\sum_{z\in E}q(x,z)\sum_{z'\in E}q(z,z')\Pi\big(\xi(z),\xi(z')\big)
\end{split}\\
&=-\E\left[\1_\sigma\circ\xi(U_{0})\Pi\big(\xi(U_2),\xi(U_3)\big)\right]+\E\left[\1_\sigma\circ\xi(U_{2})\Pi\big(\xi(U_2),\xi(U_3)\big)\right]\label{Dbar:cal1}\\
&=-\E\left[\1_\sigma\circ\xi(U_{0})\1_\sigma\circ\xi(U_2)\Pi\big(\xi(U_2),\xi(U_3)\big)\right]\notag\\
&\quad -\E\left[\1_\sigma\circ\xi(U_{0})\1_{S\setminus \{\sigma\}}\circ\xi(U_2)\Pi\big(\xi(U_2),\xi(U_3)\big)\right]\notag\\
&\quad +\E\left[\1_\sigma\circ\xi(U_{2})\Pi\big(\xi(U_2),\xi(U_3)\big)\right]\notag\\
&=\E\left[\1_{S\setminus\{\sigma\}}\circ\xi(U_{0})\1_\sigma\circ\xi(U_{2})\Pi\big(\xi(U_2),\xi(U_3)\big)\right]\notag\\
&\quad  -\E\left[\1_\sigma\circ\xi(U_{0})\1_{S\setminus \{\sigma\}}\circ\xi(U_2)\Pi\big(\xi(U_2),\xi(U_3)\big)\right].\notag
\end{align}
Here, we use the reversibility of $q$ with respect to $\pi$ to cancel
the last two terms in \eqref{Dbar:cal} and write the first term in \eqref{Dbar:cal} as the first term in \eqref{Dbar:cal1}. 
See \cite[Lemma~1 on p.8]{CMN} for the case of two types.

The proof of \eqref{eq:Dsigma1} appears in \cite[Lemma~7.1]{C:EGT}. Now \eqref{eq:Dsigma} allows for a quick proof:
$\overline{D}_1(\xi)=(b-c)\overline{f_0f_{11}}-c\overline{f_0f_{10}}-b\overline{f_1f_{01}}$. Then we use the identities $\overline{f_0f_{11}}+\overline{f_0f_{10}}=\overline{f_0f_1}$, $\overline{f_0f_{11}}+\overline{f_0f_{01}}=\overline{f_0f_{\bullet 1}}$, and $\overline{f_0f_{01}}+\overline{f_0f_{01}}=\overline{f_{01}}$. This calculation will be used in the proof of Corollary~\ref{cor:symmetric}.
\end{proof}

We are ready to state the explicit semimartingale decompositions of the density processes and identify the leading order terms. 
From \eqref{dynamics:p1}, \eqref{qwq:exp} and the martingales in \eqref{mg:hat}, we obtain the following decompositions extended from  \eqref{density:dynamics}:
\begin{align}\label{psigma:dec}
p_\sigma(\xi_t)=p_\sigma(\xi_0)+A_\sigma(t)+M_\sigma(t)=p_\sigma(\xi_0)+I_\sigma(t)+R_\sigma(t)+M_\sigma(t),
\end{align}
where
\begin{align}
I_\sigma(t)&=w\int_0^t \overline{D}_\sigma(\xi_s)\d s+\int_0^t \Bigg(\mu(\sigma)\sum_{\sigma'
\in S\setminus\{\sigma\}}p_{\sigma'}(\xi_s)-\mu(S\setminus\{\sigma\}) p_\sigma(\xi_s)\Bigg)\d s,
\label{def:I}\\
R_\sigma(t)&=w^2\sum_{x,y\in E}\pi(x)\int_0^t \big[\1_\sigma\circ\xi_{s}(y)-\1_\sigma\circ\xi_{s}(x)\big]q(x,y)R^w(x,y,\xi_s)\d s,\label{def:R}\\
\begin{split}
M_\sigma(t)&=\sum_{x,y\in E}\pi(x)\int_0^t \big[\1_\sigma\circ\xi_{s-}(y)-\1_\sigma\circ\xi_{s-}(x)\big]\d\widehat{\Lambda}_s(x,y) \\
&\quad +\sum_{x\in E}\pi(x)\int_0^t \1_{S\setminus\{\sigma\}}\circ\xi_{s-}(x)\d \widehat{\Lambda}^\sigma_s(x)\\
&\quad -\sum_{\sigma'\in S\setminus\{\sigma\}}\sum_{x\in E}\pi(x)\int_0^t \1_{\sigma}\circ\xi_{s-}(x)\d \widehat{\Lambda}^{\sigma'}_s(x).\label{def:M}
\end{split}
\end{align}
By \eqref{mg:hat}, the predictable quadratic variations and covariations of $M_\sigma$ and $M_{\sigma'}$, for $\sigma\neq \sigma'$, are 
\begin{align}
\begin{split}\label{def:<M1>}
 \langle M_\sigma,M_{\sigma}\rangle_t
&=\sum_{x,y\in E}\pi(x)^2\int_0^t \big\{\1_\sigma\circ\xi_{s}(y)[1-\1_{\sigma}\circ \xi_s(x)]\\
&\hspace{-.5cm} +[1-\1_{\sigma}\circ \xi_s(y)]\1_\sigma\circ\xi_{s}(x)\big\} q^w(x,y,\xi_{s})\d s\\
&\hspace{-.5cm}  +\sum_{x\in E}\pi(x)^2\int_0^t\big[ \1_{S\setminus\{\sigma\}}\circ\xi_{s-}(x)\mu(\sigma)+\1_\sigma\circ\xi_{s}(x)\mu\big(S\setminus\{\sigma\}\big)\big] \d s,
\end{split}\\
\begin{split}\label{def:<M2>}
\langle M_\sigma,M_{\sigma'}\rangle_t
&=-\sum_{x,y\in E}\pi(x)^2\int_0^t \big[\1_\sigma\circ\xi_{s}(y)\1_{\sigma'}\circ\xi_s(x)\\
&\quad +\1_\sigma\circ\xi_s(y)\1_{\sigma'}\circ\xi_{s}(x)\big] q^w(x,y,\xi_{s})\d s\\
& \quad -\sum_{x\in E}\pi(x)^2\int_0^t\big[ \1_{S\setminus\{\sigma\}}\circ\xi_{s-}(x) \1_{\sigma}\circ\xi_{s-}(x)\mu(\sigma)\\
&\quad +\1_{S\setminus\{\sigma'\}}\circ\xi_{s-}(x) \1_{\sigma'}\circ\xi_{s-}(x)\mu(\sigma')\big] \d s.
\end{split}
\end{align}
In Section~\ref{sec:eqn}, the above equations play the central role in characterizing the limiting density processes.

For this study, we apply the coalescing duality between $(E,q,\mu)$-voter model and the coalescing rate-$1$ $q$-Markov chains $\{B^x;x\in E\}$, where $B^x_0=x$. These chains move independently before meeting, and for any $x,y\in E$, $B^x=B^y$ after their first meeting time
$ M_{x,y}=\inf\{t\geq 0;B^x_t=B^y_t\}$. In the absence of mutation, the duality is given by
\begin{align}\label{dual:1}
\E\left[\prod_{i=1}^n \1_{\sigma_i}\circ \xi_0(B^{x_i}_t)\right]=\E_{\xi_0}\left[\prod_{i=1}^n \1_{\sigma_i}\circ\xi_t(x_i)\right] 
\end{align}
for all $\xi_0\in S^E$, $\sigma_1,\cdots,\sigma_n\in S$, distinct $x_1,\cdots,x_n \in E$ and $n\in \Bbb N$. See the proof of Proposition~\ref{prop:mutation} for the foregoing identity and the extension to the case with mutations.  

Without mutation, the density process is a martingale under the voter model by \eqref{density:dynamics}, and it follows from \eqref{def:M} and \eqref{def:<M1>} that, for any $\sigma\neq \sigma'$, 
\begin{align}\label{eq:p1p0-voter}
\E_\xi^0[p_\sigma(\xi_{t})p_{\sigma'}(\xi_{t})]=p_\sigma(\xi)p_{\sigma'}(\xi)-\nu(\1)\int_0^t \E_\xi^0[p_{\sigma\sigma'}(\xi_{s})+p_{\sigma'\sigma}(\xi_s)]\d s.
\end{align}
For the present problem, the central application of this dual relation is the foregoing identity \cite{CCC}. Let the random variables defined below \eqref{def:Well} to represent frequencies and densities be independent of the coalescing Markov chains. Then the foregoing equality implies that
\begin{align}\label{ergodic}
\P(M_{U,U'}>t)=1-\nu(\1)-2\nu(\1)\int_0^t \P(M_{V,V'}>s)\d s,\quad \forall\;t\geq 0.
\end{align}
See \cite[Corollary~4.2]{CCC} and \cite[Section~3.5.3]{AF:MC}.
This identity for meeting times has several important applications to the diffusion approximation of the voter model density processes. See \cite[Sections~3 and 4]{CCC} and \cite{CC}.

\section{Decorrelation in the ancestral lineage distributions}\label{sec:slow}
This section is devoted to a study of degenerate limits of meeting time distributions. Here, we consider meeting times defined on a sequence of spatial structures $(E_n,q^{(n)})$ as before.  According to the coalescing duality, these distributions are part of the ancestral line distributions of the voter model, and by approximation, the ancestral line distributions of the evolutionary game. On the other hand, these meeting times encode the typical local geometry of the space, but in a rough manner. With the study of these distributions, the main results of this section (Propositions~~\ref{prop:sn-selection} and \ref{prop:kell}) extend to the choice of appropriate time scaling constants and the characterization of the limiting density processes. These properties are crucial to the forthcoming limit theorems.

Our direction in this section can be outlined in more detail  as follows.
Recall the auxiliary random variables defined below \eqref{def:Well}, which are introduced to represent frequencies and densities. Under mild mixing conditions similar to those in \eqref{cond2:thetan} with $\gamma_n$ replaced by $\theta_n$ 
 and the condition $\nu_n(\1)\to 0$, the sequence $\P^{(n)}(M_{V,V'}/\gamma_n\in \cdot)$ is known to converge. The limiting distribution is a convex combination of the delta distribution at zero and an exponential distribution. Moreover, one can choose \emph{some} $s_n\to\infty$ such that $s_n/\gamma_n\to0 $ and  the following $t$-independent limit exists:
\begin{align}\label{cond:kappa0}
\overline{\kappa}_0\,\defeq\,
\lim_{n\to\infty}2\gamma_n\nu_n(\1)\P^{(n)}(M_{V,V'}>s_n t),\quad\forall\;t\in (0,\infty)
\end{align}
with $\overline{\kappa}_0=1$.  See \cite[Corollary~4.2 and Proposition~4.3]{CCC} for these results. As an extension of this existence result, our first goal in this section is to introduce \emph{sufficient} conditions for these sequences $(s_n)$. Specifically, we require that the limit \eqref{cond:kappa0} exists with $\overline{\kappa}_0\in (0,\infty)$. See Section~\ref{sec:dec1}. The following is enough for the existence and the applications in the next section.

\begin{defi}\label{def:slow}
We say that $(s_n)$ is a {\bf slow sequence} if  
\begin{align}\label{cond1:sn}
\lim_{n\to\infty}s_n=\infty,\quad \lim_{n\to\infty}\frac{s_n}{\gamma_n}=0,\quad 
\lim_{n\to\infty}\gamma_n\nu_n(\1)\e^{-ts_n}= 0,\quad \forall\;t\in (0,\infty),
\end{align}
and at least one of the two mixing conditions holds:  
\begin{align}\label{cond2:sn}
\lim_{n\to\infty}\gamma_n\nu_n(\1)\e^{-\mathbf g_n s_n}=0\quad\mbox{or}\quad \lim_{n\to\infty}\frac{\mathbf t^{(n)}_{\rm mix}}{s_n}[1+\log^+(\gamma_n/\mathbf t^{(n)}_{\rm mix})]=0.
\end{align}
\end{defi}

Our second goal is to extend the existence of the limit \eqref{cond:kappa0} to the existence of analogous time-independent limits for other meeting time distributions: for integers $\ell\geq 1$, $\ell_0,\ell_1,\ell_2\geq 0$ with $\ell_0,\ell_1,\ell_2$ all distinct, and all $t\in (0,\infty)$, 
\begin{align}\label{def:kell}
\overline{\kappa}_\ell &\,\defeq\,\lim_{n\to\infty}2\gamma_n\nu_n(\1)\P^{(n)}(M_{U_0,U_\ell}>s_n t);\\
\label{def:|kell1|ell2}
\overline{\kappa}_{(\ell_0,\ell_1)|\ell_2}&\,\defeq\,\lim_{n\to\infty}2\gamma_n\nu_n(\1)\P^{(n)}\big(M_{U_{\ell_0},U_{\ell_1}}>s_n t,M_{U_{\ell_1},U_{\ell_2}}>s_n t\big);\\
\label{def:|||}
\overline{\kappa}_{\ell_0|\ell_1|\ell_2}&\,\defeq\,\lim_{n\to\infty}2\gamma_n\nu_n(\1)\P^{(n)}\big(M_{U_{\ell_0},U_{\ell_1}}>s_n t,M_{U_{\ell_1},U_{\ell_2}}>s_n t,M_{U_{\ell_0},U_{\ell_2}}>s_n t\big).
\end{align}
The extension to $\overline{\kappa}_1$ is straightforward if we allow passing limits along subsequences. Indeed,
it follows from the definition of $\{U_\ell\}$ and $(V,V')$ that
\begin{align}\label{ineq:MUVcompare}
\frac{\pi_{\min}}{\pi_{\max}}\P(M_{V,V'}\in \Gamma)\leq \P(M_{U_0,U_1}\in \Gamma)\leq 
\frac{\pi_{\max}}{\pi_{\min}}\P(M_{V,V'}\in \Gamma),\quad\forall\; \Gamma\in \B(\R_+).
\end{align}
Hence, by taking a subsequence of $(E_n,q^{(n)})$ if necessary,
\eqref{cond:kappa0} and condition (a) of Theorem~\ref{thm:main} imply the existence of the limit $\overline{\kappa}_1$.

In Section~\ref{sec:higher-order}, we prove the existence of the other limits $\overline{\kappa}_\ell$, $\ell\geq 2$. 
More precisely, we prove tightness results as in the case of $\overline{\kappa}_1$ so that the limits may be passed along subsequences. We also prove that the limits $\overline{\kappa}_\ell$, $\ell\geq 2$, are in $(0,\infty)$. Note that in proving these results, we do not impose convergence of local geometry as in the case of discrete tori or random regular graphs.

\subsection{Mixing conditions for local meeting times}\label{sec:dec1}
To apply mixing conditions to meeting times, first, we recall some basic properties of the spectral gap and the mixing time for the product of the continuous-time $q$-Markov chains. Note that by coupling the product chain with initial condition $(x,y)$ after the two coordinates meet, we get the coalescing chain $(B^x,B^y)$ defined before \eqref{eq:p1p0-voter}.

Now, the discrete-time chain for the product chain has a transition matrix such that each of the coordinates is allowed to change with equal probability. Hence, the spectral gap is given by $\widetilde{\mathbf g}=\mathbf g/2$ \cite[Corollary~12.12 on p.161]{LPW}.
If $(\widetilde{q}_t)$ denotes the semigroup of the product chain, then
\begin{align}\label{product:bdd}
\sup_{(x,y)\in E\times E}\big\|\widetilde{q}_t\big((x,y),\cdot\big)-\pi\otimes \pi\big\|_{\rm TV}\leq 2d_{E}(t),
\end{align}
where $d_E$ is the total variation distance defined by \eqref{def:dE}. Additionally, it follows from the definition 
the mixing time in \eqref{def:tmix} that 
\begin{align}\label{ineq:tmix}
d_E(k\mathbf t_{\rm mix})\leq \e^{-k},\quad \forall\;k\in \Bbb N
\end{align}
\cite[Section~4.5 on p.55]{LPW}. By the last two displays, the analogous mixing time $\widetilde{\mathbf t}_{\rm mix}$ of the product chain satisfies 
\begin{align}\label{compare:mix} 
 \widetilde{\mathbf t}_{\rm mix}\leq 3\mathbf t_{\rm mix}.
\end{align}
We are ready to prove the first main result of Section~\ref{sec:slow}. Note that under the condition $\sup_nN_n\pi^{(n)}_{\max}<\infty$ (see the discussion below \eqref{def:wn}), the first condition in \eqref{cond2:sn} implies the first one in \eqref{sn:lim}.

\begin{prop}\label{prop:sn-selection}
Suppose that $(s_n)$ satisfies \eqref{cond1:sn} and at least one of the following mixing conditions:
\begin{align}\label{sn:lim}
\lim_{n\to\infty}\mathbf g_n s_n=\infty\quad\mbox{or}\quad \lim_{n\to\infty}\frac{\mathbf t^{(n)}_{\rm mix}}{s_n}[1+\log^+(\gamma_n/\mathbf t^{(n)}_{\rm mix})]=0.
\end{align}
Then \eqref{cond:kappa0} holds with $\overline{\kappa}_0=1$. 
\end{prop}
\begin{proof}
Write $f_n(t)=\P^{(n)}(M_{U,U'}>t)$ and $g_n(t)=\P^{(n)}(M_{V,V'}>t)$. The required result is proved in two steps. \medskip

\noindent {\bf Step 1.} We start with a preliminary result: for all $t_0\in[0,\infty)$ and $\mu\in(0,\infty)$,
\begin{align}\label{eq:Lap_nun}
&\lim_{n\to\infty}2\gamma_n\nu_n(\1)
\int_0^\infty \e^{-\mu t}g_n\big(s_n(t+t_0)\big)\d t=\frac{1}{\mu}.
\end{align}

To obtain \eqref{eq:Lap_nun}, first, we derive a representation of the integrals in \eqref{eq:Lap_nun} by $f_n(t)$. 
Note that \eqref{ergodic} under the $q^{(n)}$-chain takes the following form:
\[
f_n(s_nt)=
1-\nu_n(\1)-2\nu_n(\1)s_n\int_0^t g_n(s_ns)\d s,\quad t\geq 0.
\]
Hence,  for any fixed $0\leq t_0<\infty$,
\begin{align}\label{delta:f}
f_n(s_n(t+t_0))-f_n(s_nt_0)=-2\nu_n(\1)s_n\int_{0}^{t} g_n\big(s_n(s+t_0)\big)\d s,\quad t\geq 0.
\end{align}
Taking Laplace transforms of both sides of the last equality, we get, for $\mu>0$,
\begin{align}
\int_{0}^\infty \e^{-\mu t}\big[f_n\big(s_n(t+t_0)\big)-f_n(s_nt_0)\big]\d t&=-2\nu_n(\1)s_n\int_0^\infty \e^{-\mu t}\int_0^{t}g_n\big(s_n(s+t_0)\big)\d s\d t\notag\\
&=-\frac{2\nu_n(\1)s_n}{\mu}\int_0^\infty \e^{-\mu t}g_n\big(s_n (t+t_0)\big)\d t,\notag
\end{align}
where the last integral coincides with the integral in \eqref{eq:Lap_nun}.

Next, rewrite the last equality as
\begin{align}
&\quad 2\gamma_n\nu_n(\1)
\int_0^\infty \e^{-\mu t}g_n\big(s_n (t+t_0)\big)\d t\notag\\
&= -\frac{\gamma_n\mu}{s_n} \int_0^\infty \e^{-\mu t}\big[f_n \big(s_n(t+t_0)\big)-f_n(s_nt_0)\big]\d t\notag\\
\begin{split}
&=-\frac{\gamma_n\mu}{s_n} \int_0^\infty \e^{-\mu t}\Big\{\big[f_n \big(s_n(t+t_0)\big)-f_n(s_nt_0)\big]-
[\e^{-s_n(t+t_0)/\gamma_n}-\e^{-s_nt_0/\gamma_n}]
\Big\}\d t \\
&\quad +\frac{ \e^{-s_nt_0/\gamma_n}}{\mu+s_n/\gamma_n}.\label{eq:sn-selection0}
\end{split}
\end{align}
The last term tends to $1/\mu$ since $s_n/\gamma_n\to 0$. To take the limit of the integral term in \eqref{eq:sn-selection0}, we use the first mixing condition in \eqref{cond2:sn}. In this case, a bound for exponential approximations of the distributions of $M_{U,U'}$ \cite[Proposition~3.23]{AF:MC} gives
\begin{align}\label{eq:sn-selection1}
\begin{split}
&\left|\frac{\gamma_n\mu}{s_n} \int_0^\infty \e^{-\mu t}\Big\{\big[f_n \big(s_n(t+t_0)\big)-f_n(s_nt_0)\big]-
[\e^{-s_n(t+t_0)/\gamma_n}-\e^{-s_nt_0/\gamma_n}]
\Big\}\d t\right|\\
&\leq \frac{2}{\widetilde{\mathbf  g}_n s_n}\xrightarrow[n\to\infty]{} 0
\end{split}
\end{align}
now that $\widetilde{\mathbf g}_n=\mathbf g_n/2$.
Alternatively, by a different bound from \cite[Theorem~1.4]{Aldous:AE}, the foregoing inequality holds with the bound replaced by 
\begin{align}\label{eq:sn-selection2}
\frac{C_{\ref{eq:sn-selection2}}\widetilde{\mathbf t}^{(n)}_{\rm mix}}{s_n}\big[1+\log^+(\gamma_n/\widetilde{\mathbf t}^{(n)}_{\rm mix})\big]
\leq \frac{C_{\ref{eq:sn-selection2}}\cdot 3\mathbf t^{(n)}_{\rm mix}}{s_n}\big[1+\log^+(\gamma_n/(3\mathbf t^{(n)}_{\rm mix}))\big]
\end{align}
by \eqref{compare:mix}, the monotonicity of $x\mapsto x(1+\log (x^{-1}\vee 1))$ on $(0,\infty)$, where $C_{\ref{eq:sn-selection2}}$ is independent of the $q^{(n)}$-chains.The last term in \eqref{eq:sn-selection2} tends to zero by the second mixing condition in \eqref{cond2:sn}. 

Finally, we apply \eqref{eq:sn-selection1} and \eqref{eq:sn-selection2} to \eqref{eq:sn-selection0}. Since the last term in \eqref{eq:sn-selection0} tends to $1/\mu$, we have proved \eqref{eq:Lap_nun}. \medskip

\noindent {\bf Step 2.} We are ready to prove the existence of the limit in \eqref{cond:kappa0} and its independence of $t$.
First, note that since $g_n$ is decreasing, we have
\[
2\gamma_n\nu_n(\1)\e^{-\mu t}g_n\big(s_n(t+t_0)\big)\leq \frac{1}{t}2\gamma_n\nu_n(\1)\int_0^t \e^{-\mu s}g_n\big(s_n(s+t_0)\big)\d s,\quad\forall\;t,t_0\in (0,\infty),
\]
whereas the last integral is bounded by the same integral with the upper limit $t$ of integration replaced by $\infty$. By the \eqref{eq:sn-selection0} and the convergence proven for it in the preceding step, the last inequality implies that $t\mapsto 2\gamma_n\nu_n(\1)\e^{-\mu t}g_n(s_nt)$, $n\geq 1$, are uniformly bounded on $[a,\infty)$, for any $a\in (0,\infty)$. Hence, by Helly's selection theorem, every subsequence of $\{t\mapsto 2\gamma_n\nu_n(\1)g_n(s_nt)\}$ has a further subsequence, say indexed by $n_j$, such that for some left-continuous function $g_\infty$ on $(0,\infty)$,
\begin{align}\label{def:ginfty}
\lim_{j\to\infty}2\gamma_{n_j}\nu_{n_j}(\1)g_{n_j}(s_{n_j}t)= g_\infty(t),\quad \forall\; t\in (0,\infty). 
\end{align}
Moreover, this convergence holds boundedly on compact subsets of $ (0,\infty)$ in $t$.

To find $g_\infty$, note that, as in \eqref{eq:sn-selection1} and \eqref{eq:sn-selection2}, either of the mixing conditions \eqref{cond2:sn} implies that for fixed $0<t_1<t_2<\infty$,
\begin{align}\label{fnsn:est}
\begin{split}
\frac{\gamma_n}{s_n}[f_n(s_nt_2)-f_n(s_nt_1)]&=\frac{\gamma_n}{s_n}\big(\e^{-s_nt_2/\gamma_n}-\e^{-s_nt_1/\gamma_n}\big)+o(1)\\
&=-(t_2-t_1)+o(1),
\end{split}
\end{align}
where $o(1)$'s refer to terms tending to $0$ as $n\to\infty$. By the foregoing equality and \eqref{delta:f}, we get
\[
t_2-t_1=\lim_{j\to\infty}\int_{t_1}^{t_2}2\gamma_{n_j}\nu_{n_j}(\1)g_{n_j}(s_{n_j}t)\d t=\int_{t_1}^{t_2}g_\infty (t)\d t,\quad \forall\;0<t_1<t_2<\infty,
\]
where the last equality follows from \eqref{def:ginfty} and dominated convergence. The last equality and the left-continuity of $g_\infty$ give $g_\infty\equiv 1$ on $(0,\infty)$. We obtain \eqref{cond:kappa0} upon passing limit along $(n_j)$. By the choice of $(n_j)$, the convergence in \eqref{cond:kappa0} along the whole sequence holds. The proof is complete. 
\end{proof}

Proposition~\ref{prop:sn-selection} and the convergence in \eqref{MUU:RG}, extended to the convergence of the first moments, are enough to validate \eqref{cond:kappa0} and reinforce it to an explicit form on large random regular graphs. In Section~\ref{sec:rrg}, we give an alternative proof of these properties of \eqref{cond:kappa0}. In this case, the limit \eqref{cond:kappa0} holds only by passing subsequential limits. Nevertheless, the use of subsequences is due to the randomness of the graphs.

\subsection{Extensions}\label{sec:higher-order}
We start with a basic recursion formula to relate tail distributions of the relevant meeting times $M_{U_0,U_\ell}$, $\ell\geq 2$, to the tail distribution of $M_{U_0,U_1}$.

\begin{lem}\label{lem:MT}
For any integer $\ell\geq 1$ and $t\geq 0$,  it holds that 
\begin{align}\label{eq:shifttime0}
\begin{split}
\P(M_{U_0,U_\ell}>t)&=\e^{-2t}\P(U_0\neq U_\ell)+\int_0^t 2\e^{-2(t-s)}\P(M_{U_0,U_{\ell+1}}>s)\d s\\
&\quad -\int_0^t 2\e^{-2(t-s)}\sum_{x,y\in E}\pi(x)q^{\ell}(x,x)q(x,y)\P(M_{x,y}>s)\d s.
\end{split}
\end{align}
\end{lem}
\begin{proof}
Since $M_{x,x}\equiv 0$ and $(U_0,U_\ell)$ is independent of the meeting times, conditioning on $(U_0,U_\ell)$ gives $\P(M_{U_0,U_\ell}>t)=\P(M_{U_0,U_\ell}>t,U_0\neq U_\ell)$. Conditioning  on the first update time of $(B^{U_0},B^{U_\ell})$, which is an exponential variable with mean $1/2$, yields
\begin{align}\label{eq:st1}
\P(M_{U_0,U_\ell}>t)&=\e^{-2t}\P(U_0\neq U_\ell)+\int_0^t 2\e^{-2(t-s)}\P(U_0\neq U_{\ell},M_{U_0,U_{\ell+1}}>s)\d s.
\end{align}
Here, the initial condition $(U_0,U_{\ell+1})$ in the last term follows from transferring the first transition of state of $(B^{U_0},B^{U_\ell})$ to the initial condition. We also use the stationarity of $\{U_\ell;\ell\geq 0\}$ when that first transition is made by $B^{U_0}$. To rewrite the integral term in \eqref{eq:st1}, note that 
\begin{align*}
\P(U_0\neq U_\ell,U_0=x,U_{\ell+1}=y)
&=\P(U_0=x,U_{\ell+1}=y)-\P(U_0=U_\ell,U_0=x,U_{\ell+1}=y)\\
&=\pi(x)q^{\ell+1}(x,y)-\pi(x)q^{\ell}(x,x)q(x,y)
\end{align*}
so that
\begin{align}\label{eq:st3}
\begin{split}
\P(U_0\neq U_{\ell},M_{U_0,U_{\ell+1}}>s)=&\P(M_{U_0,U_{\ell+1}}>s)\\
&-\sum_{x,y\in E}\pi(x)q^{\ell}(x,x)q(x,y)\P(M_{x,y}>s).
\end{split}
\end{align}
Applying \eqref{eq:st3} to (\ref{eq:st1}) yields
 (\ref{eq:shifttime0}). 
\end{proof}

We are ready to prove the existence of the limits in \eqref{def:kell} and \eqref{def:|kell1|ell2}.

\begin{prop}\label{prop:kell}
For any sequence $(s_n)$ satisfying \eqref{cond1:sn}, we have the following properties:
\begin{enumerate}
\item [\hypertarget{prop:kell1}{\rm (1$\cc$)}] For any integer $\ell\geq 2$, every subsequence of $(E_n,q^{(n)})$ contains a further subsequence such that the limit in \eqref{def:kell}
exists in $[\overline{\kappa}_1,\ell\overline{\kappa}_1]$ and is independent of $t\in (0,\infty)$.

\item [\hypertarget{prop:kell2}{\rm (2$\cc$)}] Without taking any subsequence, \eqref{def:kell} holds for $\ell=2$ with $\overline{\kappa}_2=\overline{\kappa}_1$. 
\item [\hypertarget{prop:kell3}{\rm (3$\cc$)}] Suppose that \eqref{cond:q2} holds for some constant $q^{(\infty),2}$.
Then without taking any subsequence, \eqref{def:kell} holds $\overline{\kappa}_3=(1+q^{(\infty),2})\overline{\kappa}_1$.

\item [\hypertarget{prop:kell4}{\rm (4$\cc$)}] 
For all distinct nonnegative integers $\ell_0,\ell_1,\ell_2$, it holds that 
\[
\ok_{(\ell_1,\ell_2)|\ell_0}+\ok_{(\ell_0,\ell_1)|\ell_2}-\ok_{\ell_0|\ell_1|\ell_2}=\ok_{|\ell_2-\ell_0|},
\] 
provided that all of the limits defining these constants exist.

\end{enumerate}
\end{prop}
\begin{proof} 
(1$\cc$) To lighten notation in the rest of this proof but only in this proof, write $A_\ell=\P(U_0\neq U_\ell)$, $J_\ell$ for $M_{U_0,U_\ell}$, 
\[
B_\ell=\sum_{x,y\in E}\pi(x)q^\ell(x,x)q(x,y),
\]
and $K_\ell$ for the first meeting time for the pair of coalescing Markov chains where the initial condition is distributed independently as $B_\ell^{-1}\pi(x)q^\ell(x,x)q(x,y)$ provided that $B_\ell\neq 0$. We set $K_\ell$ to be an arbitrary random variable.

Fix an integer $\ell\geq 1$. If $\mathbf e$ is an independent exponential variable with mean $1$, then \eqref{eq:shifttime0} can be written as 
\[
\P(J_\ell>t)=A_\ell\P(\tfrac{1}{2} \mathbf e>t)+\P(J_{\ell+1}+\tfrac{1}{2} \mathbf e>t,\tfrac{1}{2} \mathbf e\leq t)-B_\ell\P(K_{\ell}+\tfrac{1}{2} \mathbf e>t,\tfrac{1}{2} \mathbf e\leq t).
\]
After rearrangement, the foregoing equality yields
\begin{align*}
\P(J_{\ell+1}+\tfrac{1}{2} \mathbf e>t)
&=\P(J_\ell>t)+B_\ell\P(K_{\ell}+\tfrac{1}{2} \mathbf e>t)+(1-A_\ell-B_\ell)\P(\tfrac{1}{2} \mathbf e>t).
\end{align*}
Hence, for all left-open intervals $\Gamma\subset (0,\infty)$, 
\begin{align}\label{id:JK}
\begin{split}
&\quad\, \P(J_{\ell+1}+\tfrac{1}{2} \mathbf e\in \Gamma)+(A_\ell+B_\ell)\P(\tfrac{1}{2} \mathbf e\in \Gamma)\\
&=\P(J_\ell\in \Gamma)+B_\ell\P(K_{\ell}+\tfrac{1}{2} \mathbf e\in \Gamma)+\P(\tfrac{1}{2} \mathbf e\in \Gamma).
\end{split}
\end{align}
Since $q^{\ell}(x,x)\leq 1$,
we have $B_\ell\P(K_\ell\in \cdot)\leq \P(J_1\in \cdot)$, and so, the foregoing identity gives
\begin{align}\label{id:JK1}
\begin{split}
&\quad \P(J_{\ell+1}+\tfrac{1}{2} \mathbf e\in \Gamma)+(A_\ell+B_\ell)\P(\tfrac{1}{2} \mathbf e\in \Gamma)\\
&\leq \P(J_\ell\in \Gamma)+\P(J_{1}+\tfrac{1}{2} \mathbf e\in \Gamma)+\P(\tfrac{1}{2} \mathbf e\in \Gamma).
\end{split}
\end{align}

We are ready to prove the required result. 
For any $0<T_0<T_1< \infty$, repeated applications of the first and third limits in \eqref{cond1:sn} for all of the next three equalities give
\begin{align}
&\quad\, \limsup_{n\to\infty}2\gamma_n\nu_n(\1) \P^{(n)}\big(J_{\ell+1}\in (s_nT_0,s_nT_1]\big)\notag\\
&\leq  \limsup_{n\to\infty}2\gamma_n\nu_n(\1) \P^{(n)}\big(J_{\ell+1}+\tfrac{1}{2}\mathbf e\in (s_nT_0,s_n2T_1]\big)\notag\\
&\leq \limsup_{n\to\infty}2\gamma_n\nu_n(\1)\P^{(n)}\big(J_{\ell}\in (s_nT_0,s_n2T_1]\big)\notag\\
&\quad +\limsup_{n\to\infty}2\gamma_n\nu_n(\1)\P^{(n)}\big(J_{1}+\tfrac{1}{2}\mathbf e\in (s_nT_0,s_n2T_1]\big)\label{eq:kell0}\\
&\leq \limsup_{n\to\infty}2\gamma_n\nu_n(\1)\P^{(n)}\big(J_{\ell}\in (s_nT_0,s_n2T_1]\big)\notag\\
&\quad +\limsup_{n\to\infty}2\gamma_n\nu_n(\1)\P^{(n)}\big(J_{1}\in (s_n2^{-1} T_0,s_n2T_1]\big)\notag\\
&\leq (\ell+1)\limsup_{n\to\infty}2\gamma_n\nu_n(\1)\P^{(n)}\big(J_{1}\in (s_n2^{-\ell}T_0,s_n2^{\ell}T_1]\big)=0,\label{eq:kell1}
\end{align}
where \eqref{eq:kell0} also uses \eqref{id:JK1}, the last inequality follows from induction, and the equality in \eqref{eq:kell1} follows from \eqref{def:kell} with $\ell=1$. Moreover, by setting $T_0=t$ and $T_1=\infty$, a similar argument as in the display for \eqref{eq:kell1} shows that \eqref{def:kell} with $\ell=1$ gives
\begin{align}\label{eq:kell3}
\limsup_{n\to\infty}2\gamma_n\nu_n(\1)\P^{(n)}(J_{\ell+1}>s_n a)\leq (\ell+1) \overline{\kappa}_1,\quad \forall\;t\in (0,\infty).
\end{align}
On the other hand,
since $\P(J_{\ell+1}+\tfrac{1}{2} \mathbf e\in \cdot)+|A_\ell+B_\ell-1|\P(\tfrac{1}{2}\mathbf e\in \cdot)\geq \P(J_\ell\in \cdot)$ by \eqref{id:JK}, it follows from \eqref{def:kell} with $\ell=1$ and an argument similar to the one leading to \eqref{eq:kell1} that 
\begin{align}
\liminf_{n\to\infty}2\gamma_n\nu_n(\1)\P^{(n)}(J_{\ell+1}>s_n t)\geq \overline{\kappa}_1>0,\quad\forall\;t\in (0,\infty).\label{eq:kell4}
\end{align}

Combining \eqref{eq:kell3} and \eqref{eq:kell4}, we deduce that for fixed $t_0\in (0,\infty)$, any subsequence of the numbers
$2\gamma_n\nu_n(\1)\P^{(n)}(J_{\ell+1}>s_nt_0)$ has a further subsequence that converges in $[\overline{\kappa}_1,(\ell+1)\overline{\kappa}_1]$. By \eqref{eq:kell1}, this limit extends to the existence of the limit of the corresponding subsequence of 
$2\gamma_n\nu_n(\1)\P^{(n)}(J_{\ell+1}>s_nt)$ 
for any $t\in (0,\infty)$, and all of these limits for different $t$ are equal.
We have proved \eqref{def:kell}. \medskip

\noindent (2$\cc$) Note that $B_1^{(n)}=0$ since $\tr(q^{(n)})=0$ by assumption.  
Then an inspection of \eqref{eq:kell0} shows the second limit superior on the right-hand side there can be dropped. The rest of the argument in (2$\cc$), especially \eqref{eq:kell3} and \eqref{eq:kell4}, can be adapted accordingly to get
the required identity. \medskip
 
\noindent (3$\cc$) The proof is done again by improving the argument for \eqref{eq:kell3} and \eqref{eq:kell4}, but now using \eqref{id:JK} with $\ell=2$. In doing so, we also use the following implication of \eqref{cond:q2}:
\[
\lim_{n\to\infty}\sup_{s\geq 0}\gamma_n\nu_n(\1)\big|B^{(n)}_2\P^{(n)}(K_2>s)-q^{(\infty),2}\P^{(n)}(J_1>s)\big|=0,
\]
which follows since the distributions of $K_2$ and $J_1$ differ by the initial conditions. \medskip 

\noindent (4$\cc$) By the definitions in \eqref{def:kell}--\eqref{def:|||}, we have
\begin{align*}
&\eqspace (\ok_{(\ell_1,\ell_2)|\ell_0}-\ok_{\ell_0|\ell_1|\ell_2})+(\ok_{(\ell_0,\ell_1)|\ell_2}-\ok_{\ell_0|\ell_1|\ell_2})+\ok_{\ell_0|\ell_1|\ell_2}\\
&=\lim_{n\to\infty}2\gamma_n\nu_n(\1)\P^{(n)}(M_{U_{\ell_0},U_{\ell_1}}>s_n t,M_{U_{\ell_1},U_{\ell_2}}\leq s_n t,M_{U_{\ell_0},U_{\ell_2}}>s_n t)\\
&\quad\, +\lim_{n\to\infty}2\gamma_n\nu_n(\1) \P^{(n)}(M_{U_{\ell_0},U_{\ell_1}}\leq s_n t,M_{U_{\ell_1},U_{\ell_2}}> s_n t,M_{U_{\ell_0},U_{\ell_2}}>s_n t)\\
&\quad \,+\lim_{n\to\infty}2\gamma_n\nu_n(\1)\P^{(n)}(M_{U_{\ell_0},U_{\ell_1}}>s_n t,M_{U_{\ell_1},U_{\ell_2}}>s_nt,M_{U_{\ell_0},U_{\ell_2}}>s_nt)\\
&=\lim_{n\to\infty}2\gamma_n\nu_n(\1)\P^{(n)}(M_{U_{\ell_0},U_{\ell_2}}>s_n t)=\ok_{|\ell_2-\ell_0|}.
\end{align*}
Here, the next to the last equality follows since on $\{M_{U_{\ell_0},U_{\ell_2}}>s_n t\}$, we cannot have both $M_{U_{\ell_0},U_{\ell_1}}\leq s_n t$ and $M_{U_{\ell_1},U_{\ell_2}}\leq s_n t$ by the coalescence of the Markov chains, and the last equality follows from the stationarity of the chain $\{U_\ell\}$. The proof is complete.
\end{proof}

We close this subsection with another application of Lemma~\ref{lem:MT}. It will be used in Section~\ref{sec:eqn}.

\begin{prop}\label{prop:Mcompare}
Let $s_0\in (2,\infty)$.
For all integers $\ell\geq 1$ and all $t\in (0,\infty)$, it holds that 
\begin{align}\label{ineq:Cell}
\int_0^{t} \P(M_{U_0,U_\ell}>s_0s)\d s\leq \sum_{j=1}^\ell \prod_{k=1}^{j-1} \big(1-\e^{-2^{k+1}t}\big)^{-1}\int_0^{2j t} \P(M_{U_0,U_1}>s_0s)\d s,
\end{align}
where $\prod_{k=i}^ja_k\equiv 1$ for $j<i$. 
\end{prop}
\begin{proof}
We prove \eqref{ineq:Cell} by an induction on $\ell\geq 1$. The inequality is obvious for all $t\in (0,\infty)$ if $\ell=1$.

Suppose that for some $\ell\geq 1$, \eqref{ineq:Cell} holds for all $t\in (0,\infty)$. By \eqref{eq:shifttime0} with $t$ replaced by $s_0r$, 
\begin{align}\label{Mcompare:1}
\begin{split}
&\eqspace\int_0^{r} 2s_0 \e^{-2s_0 (r-s)}\P(M_{U_0,U_{\ell+1}}>s_0s)\d s\\
&\leq 
\P(M_{U_0,U_\ell}>s_0 r)+\int_0^{ r}2s_0 \e^{-2s_0 ( r-s)}\P(M_{U_0,U_1}>s_0s)\d s.
\end{split}
\end{align}
The assumption $s_0\in (2,\infty)$ gives $t\leq 2t(1-s_0^{-1})$, and so, for $h$ nonnegative and Borel measurable, 
\begin{align}
\begin{split}\label{Mcompare:2}
 &\eqspace (1-\e^{-4t})\int_0^{t}h(s_0 s)\d s\leq \int_0^{2t(1-s_0^{-1})}h(s_0 s)(1-\e^{-4t})\d s\\
 &\leq \int_0^{2t} h(s_0 s)\big(1-\e^{-2s_0(2t-s)}\big)\d s=\int_0^{2t}\int_0^r 2s_0 \e^{-2s_0(r- s)}h(s_0 s)\d s\d r \\&\leq \int_0^{2t} h(s_0s)\d s.
 \end{split}
\end{align}
Integrating both sides of \eqref{Mcompare:1} over $[0,2t]$ and applying the first and last inequalities in \eqref{Mcompare:2} give
\begin{align}
&\eqspace(1-\e^{-4t})
\int_0^{t} \P(M_{U_0,U_{\ell+1}}>s_0 s)\d s\notag\\
&\leq 
\int_0^{2t} \P(M_{U_0,U_\ell}>s_0 s)\d s+\int_0^{2t} \P(M_{U_0,U_1}>s_0 s)\d s \notag\\
&\leq \sum_{j=1}^\ell \prod_{k=1}^{j-1} \big(1-\e^{-2^{k+2}t}\big)^{-1}\int_0^{2^{j+1} t} \P(M_{U_0,U_1}>s_0s)\d s +\int_0^{2t} \P(M_{U_0,U_1}>s_0 s)\d s\notag\\
&\leq \sum_{j=2}^{\ell+1} \prod_{k=2}^{j-1} \big(1-\e^{-2^{k+1}t}\big)^{-1}\int_0^{2^{j} t} \P(M_{U_0,U_1}>s_0s)\d s +\int_0^{2t} \P(M_{U_0,U_1}>s_0 s)\d s,\label{Mcompare:final}
\end{align}
where the second inequality follows from induction. Dividing both sides of \eqref{Mcompare:final} by $(1-\e^{-4t})$ proves \eqref{ineq:Cell} for $\ell$ replaced by $\ell+1$. Hence, \eqref{ineq:Cell} holds for all $\ell\geq 1$ by induction. 
\end{proof}

\section{Convergence of the vector density processes}\label{sec:eqn}
We present the proofs of Theorem~\ref{thm:main} and Corollary~\ref{cor:symmetric} in this section. The key result is Proposition~\ref{prop:duhamel} where we reduce the evolutionary game model to the voter model.  {\bf Throughout this section, 
  conditions (a)--(d) of Theorem~\ref{thm:main} are in force.}

 The other settings for this section are as follows. First, we write $I^{(n)}_\sigma=I_\sigma(\theta_n t)$ for the process $I_\sigma(t)$ defined by \eqref{def:I}, when the underlying particle system is based on $(E_n,q^{(n)})$. This notation extends to the other processes in the decompositions \eqref{psigma:dec} by using the same time change. Next, recall that $S$ denotes the type space.  We will mostly consider  $(\sigma_0,\sigma_2,\sigma_3)\in S\times S\times S$ such that $\sigma_0\neq \sigma_2$. These triplets fit into the context of \eqref{eq:Dsigma}, from which we will prove the limiting replicator equation in Theorem~\ref{thm:main}. Additionally, given an admissible sequence $(\theta_n,\mu_n,w_n)$ such that $\lim_n \theta_n/\gamma_n=0$, we can choose a slow sequence $(s_n)$ (recall Definition~\ref{def:slow}) such that 
\begin{align}
\lim_{n\to\infty}\frac{s_n}{\theta_n}=0.\label{sn:adm}
\end{align}

\subsection{Asymptotic closure of equations and path regularity}\label{sec:closure}
We begin by showing that the leading order drift term $I_\sigma^{(n)}$ in \eqref{psigma:dec} can be asymptotically closed by the vector density process $(p_{\sigma}(\xi_{\theta_nt});\sigma\in S)$. By \eqref{def:I}, this term takes the following explicit form: 
\begin{align}\label{def:In}
\begin{split}
I^{(n)}_\sigma(t)&=w_n\theta_n\int_0^t \overline{D}_\sigma(\xi_{\theta_n s})
\d s\\
&+\int_0^t \Bigg(\theta_n\mu_n(\sigma)[1-p_{\sigma}(\xi_{\theta_ns})]-\theta_n\mu_n(S\setminus\{\sigma\}) p_\sigma(\xi_{\theta_ns})\Bigg)\d s.
\end{split}
\end{align}
Specifically, in terms of the explicit form of $\overline{D}_\sigma$ in \eqref{eq:Dsigma}, our goal is to prove that 
\begin{align}\label{closure}
\lim_{n\to\infty}\sup_{\xi\in S^{E_n}}\E^{w_n}_{\xi}\left[\left|\int_0^tw_n\theta_n \overline{f_{\sigma_0} f_{\sigma_2\sigma_3}}(\xi_{\theta_ns})-w_\infty Q_{\sigma_0,\sigma_2\sigma_3}\big(p(\xi_{\theta_ns })\big)\d s \right|\right]=0,
\end{align}
where $\sigma_0\neq \sigma_2$, $w_\infty$ is defined by \eqref{def:wn}, and $Q_{\sigma_0,\sigma_2\sigma_3}(X)$ is a polynomial in $X=(X_\sigma)_{\sigma\in S}$ defined by
\begin{align}
\begin{split}
\textcolor{black}{Q_{\sigma_0,\sigma_2\sigma_3}(X)}&\defeq\1_{\{\sigma_2=\sigma_3\}}(\ok_{(2,3)|0}-\overline{\kappa}_{0|2|3})X_{\sigma_0}X_{\sigma_2}\\
&\quad +\1_{\{\sigma_0=\sigma_3\}}(\overline{\kappa}_{(0,3)|2}-\overline{\kappa}_{0|2|3})X_{\sigma_0}X_{\sigma_2}\\
&\quad +\overline{\kappa}_{0|2|3}X_{\sigma_0}X_{\sigma_2}X_{\sigma_3}.
\end{split}
\label{def:Qsigma}
\end{align}
The choice of $Q_{\sigma_0,\sigma_2\sigma_3}$ is due to the proof of Lemma~\ref{lem:RW}.

The proof of \eqref{closure} begins with an inequality central to the proof of \cite[Theorem~2.2]{CCC}, which goes back to \cite{CMP} and is also central to the proof of \cite[Lemma~4.2]{CC}. This inequality is presented in a general form  for future references. In what follows, we write $a\wedge b$ for $\min\{a,b\}$.

\begin{prop}\label{prop:L2}
Given a Polish space $E_0$ and $T\in (0,\infty)$, let $(X_t)_{0\leq t\leq T}$ be an $E_0$-valued Markov process with c\'adl\'ag paths. Let $f$ and $g$ be bounded Borel measurable functions defined on $E_0$. Suppose that $x\mapsto \E_x[f(X_t)]$ is Borel measurable, and for some bounded decreasing function $a(t)$, 
\begin{align}\label{def:a(t)}
\sup_{x\in E_0}\E_x[|f(X_t)|]\leq  a(t),\quad \forall\;t\in [0,T]. 
\end{align}
Then for all $0<2\delta<t\leq T$,
\begin{align}\label{ineq:L2}
\begin{split}
&\sup_{x\in E_0}\E_x\left[\left|\int_0^t\big(f(X_s)-g(X_s)\big)\d s\right|\right]\\
& \leq \int_0^\delta a(s)\d s + \left(8ta(\delta)\int_0^\delta a(s)\d s\right)^{1/2}+3\delta \|g\|_\infty+t\sup_{x\in E_0}\big|\E_x[f(X_\delta)]-g(x)\big|.
\end{split}
\end{align}
\end{prop}
\begin{proof}
For $s\geq \delta$, define $H(s)=f(X_s)-\E_{X_{s-\delta}}[f(X_\delta)]$. Then
\begin{align}
&\quad \E_x\left[\left|\int_0^t\big(f(X_s)-g(X_s)\big)\d s\right|\right]\notag\\
&\leq 
\int_0^\delta \big(\E_x[|f(X_s)|]+\|g\|_\infty\big) \d s+\E_x\left[\left(\int_\delta^t H(s)\d s\right)^2\right]^{1/2}\notag\\
&\quad\,+\E_x\left[\int_\delta^t|
\E_{X_{s-\delta}}[f(X_\delta)]-g(X_{s-\delta})|\d s \right]
+\E\left[\left|\int_\delta^t \big(g(X_{s-\delta})-g(X_s)\big)\d s\right|\right]\notag\\
\begin{split}
&\leq \int_0^\delta a(s)\d s +\E_x\left[\left(\int_\delta^t H(s)\d s\right)^2\right]^{1/2}+3\delta \|g\|_\infty\\
&\quad + t\sup_{x\in E_0}\big|\E_x[f(X_\delta)]-g(x)\big|.\label{H0}
\end{split}
\end{align}
Note that in the last inequality, $2\delta\|g\|_\infty$ is contributed by the integral $\int_\delta^t \big(g(X_{s-\delta})-g(X_s)\big)\d s$.

To bound the second term in \eqref{H0}, we note that for $\delta\leq r<s-\delta$, 
\begin{align*}
\E_x[H(s)H(r)]
&=\E_x\Big[f(X_s)\Big(f(X_r)-\E_{X_{r-\delta}}[f(X_\delta)]\Big)\Big]\\
&\quad\,-\E_x\Big[\E_{X_{s-\delta}}[f(X_\delta)]
\Big(f(X_r)-\E_{X_{r-\delta}}[f(X_\delta)]\Big)\Big]
=0,
\end{align*}
where the last equality follows by applying the Markov property at time $s-\delta$ to the first expectation on the right-hand side of the first equality. Hence,
\begin{align}
 \E_x\left[\left(\int_\delta^t H(s)\d s\right)^2\right]&=2\int_\delta^t \int_r^{t\wedge (r+\delta)}\E_x[H(r)H(s)]\d s \d r,\label{H1}
\end{align}
whereas for $r\leq s\leq r+\delta$, 
\begin{align}
&\quad \E_x[H(r)H(s)]\notag\\
&=\E_x[f(X_r)f(X_s)]-\E_x\big[f(X_r)\E_{X_{s-\delta}}[f(X_\delta)]\big]\notag\\
&\quad -\E_x\big[\E_{X_{r-\delta}}[f(X_\delta)]f(X_s)\big] +\E_x\big[\E_{X_{r-\delta}}[f(X_\delta)]\E_{X_{s-\delta}}[f(X_\delta)]\big]\notag\\
&\leq a(r)a(s-r)+a(r)a(\delta)+a(\delta)a(s)+a(\delta)^2\label{H2-0}
\end{align}
by \eqref{def:a(t)} and the Markov property.
Since $a$ is decreasing, integrating the terms in the last line yields
\begin{align}
&\eqspace 2\int_\delta^t \int_r^{t\wedge (r+\delta)}\big(a(r)a(s-r)+a(r)a(\delta)+a(\delta)a(s)+a(\delta)^2\big)\d s\d r\notag\\
&\leq 2ta(\delta)\int_0^\delta a(s)\d s+2a(\delta)^2t \delta +2ta(\delta)\int_0^\delta a(s)\d s+2a(\delta)^2 t\delta\notag\\
&\leq 8ta(\delta)\int_0^\delta a(s)\d s.\label{H2}
\end{align}
Applying \eqref{H1}--\eqref{H2} to \eqref{H0}, we get \eqref{ineq:L2}.  
\end{proof}

To prove \eqref{closure}, we apply Proposition~\ref{prop:L2} with the following choice:
\begin{align}\label{setup}
\begin{split}
X_t&=\xi_{\theta_nt} \mbox{ under }\P^{w_n}_{\xi},\\
 \delta&=\delta_n=2s_n/\theta_n,\\
f&=f_n=w_n\theta_n \overline{f_{\sigma_0} f_{\sigma_2\sigma_3}},\\
 g&=g_n=w_\infty Q_{\sigma_0,\sigma_2\sigma_3}\circ p.
\end{split}
\end{align}
The next two results are used to identify the appropriate $a(t)=a_n(t)$ that satisfies \eqref{def:a(t)}. We recall for the last time that conditions (a)--(d) of Theorem~\ref{thm:main} are in force throughout Section~\ref{sec:eqn}.

\begin{lem}\label{lem:UVbdd}
Let $s_0\in (2,\infty)$. Then for any  $t\in(0,\infty)$ and integer $\ell\geq 1$, 
\begin{align}
\begin{split}
&2\gamma_n\nu_n(\1)\int_0^{t}\P^{(n)}(M_{U_0,U_\ell}>s_0 s)\d s\\
\leq & 
C_{\ref{UVbdd:1}}\Bigg(\sum_{j=1}^\ell \prod_{k=1}^{j-1} \big(1-\e^{-2^{k+1}t}\big)^{-1}\Bigg)\left(\frac{\pi^{(n)}_{\max}}{\pi^{(n)}_{\min}}\right) \left[\ell t+\min\Bigg\{ \frac{1}{\mathbf g_ns_0}, \frac{\mathbf t^{(n)}_{\rm mix}}{s_0}[1+\log^+(\gamma_n/\mathbf t^{(n)}_{\rm mix})]\Bigg\}\right],\label{UVbdd:1}
\end{split}
\end{align}
where $C_{\ref{UVbdd:1}}$ is a universal constant.
\end{lem}
\begin{proof}
By \eqref{ineq:MUVcompare} and Proposition~\ref{prop:Mcompare}, we obtain the following inequality:
\begin{align}
&\eqspace 2\gamma_n\nu_n(\1)\int_0^{t}\P^{(n)}(M_{U_0,U_\ell}>s_0 s)\d s\notag\\
&\leq\frac{\gamma_n}{s_0}\cdot  \Bigg(\sum_{j=1}^\ell \prod_{k=1}^{j-1} \big(1-\e^{-2^{k+1}t}\big)^{-1}\Bigg)\left(\frac{\pi^{(n)}_{\max}}{\pi^{(n)}_{\min}}\right)
2s_0\nu_n(\1)\int_0^{2\ell t}\P^{(n)}(M_{V,V'}>s_0 s)\d s\notag\\
\begin{split}
&=\frac{\gamma_n}{s_0} \cdot \Bigg(\sum_{j=1}^\ell \prod_{k=1}^{j-1} \big(1-\e^{-2^{k+1}t}\big)^{-1}\Bigg)\left(\frac{\pi^{(n)}_{\max}}{\pi^{(n)}_{\min}}\right)\\
&\quad \times \left[\P^{(n)}(M_{U,U'}>0)-\P^{(n)}(M_{U,U'}>2\ell s_0 t)\right]\label{UVbdd:00}
\end{split}\\
\begin{split}
&\leq C_{\ref{UVbdd:0}}\frac{\gamma_n}{s_0}\cdot \Bigg(\sum_{j=1}^\ell \prod_{k=1}^{j-1} \big(1-\e^{-2^{k+1}t}\big)^{-1}\Bigg)\left(\frac{\pi^{(n)}_{\max}}{\pi^{(n)}_{\min}}\right)\\
&\eqspace\times \left[\big(1-\e^{-2\ell s_0t/\gamma_n}\big)+\min\Bigg\{ \frac{2}{\widetilde{\mathbf g}_n\gamma_n}, \frac{\widetilde{\mathbf t}^{(n)}_{\rm mix}}{\gamma_n}\big[1+\log^+(\gamma_n/\widetilde{\mathbf t}^{(n)}_{\rm mix})\big]\Bigg\}\right] \label{UVbdd:0}
\end{split}
\end{align}
for a universal constant $C_{\ref{UVbdd:0}}$. Here, \eqref{UVbdd:00} follows from \eqref{ergodic}, and \eqref{UVbdd:0} follows from the exponential approximation of $M_{U,U'}$ as in the proof of Proposition~\ref{prop:sn-selection}. Recall the reduction of mixing of products chains to mixing of the coordinates as used in that proposition, and the inequality $1-\e^{-x}\leq x$ holds for all $x\geq 0$. Hence, we obtain \eqref{UVbdd:1} from \eqref{UVbdd:0}. The proof is complete. 
\end{proof}

\begin{prop}\label{prop:duhamel}
Fix $(\sigma_0,\sigma_1,\sigma_2,\sigma_3)\in S\times S\times S$ such that $\sigma_0\neq \sigma_2$ and $\sigma_0\neq \sigma_1$. \medskip 

\noindent {\rm (1$\cc$)}
For any $w\in [0,\overline{w}]$ and $t\in(0,\infty)$, the following estimates of the evolutionary game by the voter model holds: for some constant $C_{\ref{LwL:0}}$ depending only on $\Pi$,
\begin{align}
&\quad 
\sup_{\xi\in S^{E}}\Big|\E^{w}_\xi\big[\,\overline{f_{\sigma_0} f_{\sigma_2\sigma_3}}(\xi_{t})\big]-\E^0_\xi\big[\,\overline{f_{\sigma_0} f_{\sigma_2\sigma_3}}(\xi_{t})\big]\Big|\notag\\
\begin{split}
&\leq C_{\ref{LwL:0}} w\int_0^t \P(M_{U_0,U_2}>s)\d s
+C_{\ref{LwL:0}}w\mu(\1)\int_0^t\int_0^s\P(M_{U_0,U_2}>r)\d r\d s;\label{LwL:0}
\end{split}\\
&\quad 
\sup_{\xi\in S^{E}}\Big|\E^{w}_\xi\big[\,\overline{ f_{\sigma_0\sigma_1}}(\xi_{t})\big]-\E^0_\xi\big[\,\overline{ f_{\sigma_0\sigma_1}}(\xi_{t})\big]\Big|\notag\\
\begin{split}
&\leq C_{\ref{LwL:0}} w\int_0^t \P(M_{U_0,U_1}>s)\d s
+C_{\ref{LwL:0}}w\mu(\1)\int_0^t\int_0^s\P(M_{U_0,U_1}>r)\d r\d s.\label{LwL:001}
\end{split}
\end{align}
 \medskip 

\noindent {\rm (2$\cc$)} For any admissible sequence $(\theta_n,\mu_n,w_n)$ and $T\in(0,\infty)$, it holds that 
\begin{align}\label{LwL:1}
\begin{split}
&\lim_{n\to\infty}\int_0^T\sup_{\xi\in S^{E_n}}\left|\E^{w_n}_{\xi}\left[w_n\theta_n \overline{f_{\sigma_0} f_{\sigma_2\sigma_3}}(\xi_{2s_n t})\right]-\E^{0}_{\xi}\left[w_n\theta_n\overline{f_{\sigma_0} f_{\sigma_2\sigma_3}}(\xi_{2s_n t})\right]\right|\d t=0.
\end{split}
\end{align}
\end{prop}
\begin{proof}
(1$\cc$) Recall that the generator of $\mathsf L^w$ of the evolutionary game is given by \eqref{def:Lw}, and $\mathsf L=\mathsf L^0$ denotes the generator of the voter model. By Duhamel's principle \cite[(2.15) in Chapter~1]{EK:MP},
\begin{align}\label{expansion}
\e^{t\mathsf L^w}H=\e^{t\mathsf L }H+\int_0^t\e^{(t-s)\mathsf L^w}(\mathsf L^w-\mathsf L)\e^{s\mathsf L}H\d s .
\end{align}
Here, it follows from \eqref{def:Lw} that
\begin{align}
(\mathsf L^w-\mathsf L)H_1(\xi)&=\sum_{x,y\in E}[q^w(x,y,\xi)-q(x,y)][H_1(\xi^{x,y})-H_1(\xi)].
\label{LwL1}
\end{align}
To apply \eqref{expansion} and \eqref{LwL1}, we choose $H=\overline{f_{\sigma_0} f_{\sigma_2\sigma_3}}$ and $H_1=\e^{s\mathsf L}\overline{f_{\sigma_0} f_{\sigma_2\sigma_3}}$. The following bound will be proved in Proposition~\ref{prop:mutation} (2$\cc$):
\begin{align}\label{claim:w0}
\begin{split}
&\sup_{\xi\in S^E}|\e^{s\mathsf L}\overline{f_{\sigma_0} f_{\sigma_2\sigma_3}}(\xi^x)-\e^{s\mathsf L}\overline{f_{\sigma_0} f_{\sigma_2\sigma_3}}(\xi)|\\
&\leq 
\sum_{\ell\in \{0,2,3\}}4\P(M_{U_0,U_2}>s,B^{U_\ell}_{s}= x)\\
&\quad +\sum_{\ell\in \{0,2,3\}}4\mu(\1)\int_0^s \P(M_{U_0,U_2}>r,B^{U_\ell}_s=x)\d r.
\end{split}
\end{align}

To bound $(\mathsf L^w-\mathsf L)\e^{s\mathsf L}H=(\mathsf L^w-\mathsf L)H_1$ in the expansion \eqref{expansion}, notice that 
\begin{align}\label{qwq:bdd}
|q^w(x,y,\xi)-q(x,y)|\leq  C_{\ref{qwq:bdd}}wq(x,y)
\end{align}
by \eqref{qwq:exp} for some $C_{\ref{qwq:bdd}}$ depending only on $\Pi$. Putting \eqref{LwL1}, \eqref{claim:w0} and \eqref{qwq:bdd} together, we get
\begin{align*}
&\quad \sup_{\xi\in S^E}| (\mathsf L^w-\mathsf L)\e^{s \mathsf L}\overline{f_{\sigma_0} f_{\sigma_2\sigma_3}}(\xi)|\\
&\leq C_{\ref{qwq:bdd}} 4w\sum_{\ell\in \{0,2,3\}}\sum_{x,y\in E}q(x,y)\P(M_{U_0,U_2}>s,B^{U_\ell}_s=x)\\
&\eqspace +C_{\ref{qwq:bdd}} 4w\mu(\1)\sum_{\ell\in \{0,2,3\}}\sum_{x,y\in E}q(x,y)\int_0^s \P(M_{U_0,U_2}>r,B^{U_\ell}_s=x)\d r\\
&\leq C_{\ref{qwq:bdd}}12 w\P(M_{U_0,U_2}>s)+C_{\ref{qwq:bdd}}12 w\mu(\1)\int_0^s \P(M_{U_0,U_2}>r)\d r.
\end{align*}
Since $\e^{(t-s)\mathsf L^w}$ is a probability, the required inequality in \eqref{LwL:1} follows upon applying the foregoing inequality to \eqref{expansion}. We have proved \eqref{LwL:0}. The proof of \eqref{LwL:001} is almost the same if we use Proposition~\ref{prop:mutation} (3$\cc$) instead of Proposition~\ref{prop:mutation} (2$\cc$). The details are omitted.   \medskip 

\noindent (2$\cc$) By the first limit in \eqref{def:wn} and \eqref{LwL:0}, it is enough to show that all of the following limits hold:
\begin{align}
\label{problem:wn1}
&\lim_{n\to\infty}
w_n(2s_n)\cdot 2\gamma_n\nu_n(\1)\int_0^{T} \P^{(n)}(M_{U_0,U_2}>2s_n s)\d s=0;\\
&\lim_{n\to\infty}
[w_n(2s_n)+1]\cdot \mu_n(\1)(2s_n)\cdot   2\gamma_n\nu_n(\1)\int_0^{T}\P^{(n)}(M_{U_0,U_2}>2s_n s)\d s=0\label{problem:wn2}.
\end{align}
(The limit \eqref{problem:wn2} is stronger than needed but is convenient for the other proofs below.)

To get \eqref{problem:wn1}, first, note that by \eqref{cond2:sn}, \eqref{sn:adm} and the limit superior in \eqref{def:wn},
\begin{align}\label{problem:wn1-1}
\lim_{n\to\infty}w_n(2s_n)\cdot  \left[T+\min\Bigg\{ \frac{1}{\mathbf g_n(2s_n)}, \frac{\mathbf t^{(n)}_{\rm mix}}{2s_n}[1+\log^+(\gamma_n/\mathbf t^{(n)}_{\rm mix})]\Bigg\}\right]=0.
\end{align}
We get \eqref{problem:wn1} from applying \eqref{cond:pi} and \eqref{problem:wn1-1} to \eqref{UVbdd:1} with $s_0=2s_n$.
For \eqref{problem:wn2}, 
 $\lim_n\mu_n(\1)(2s_n)=0$ by \eqref{def:mun} and \eqref{sn:adm}. The limit superior in \eqref{def:wn}
 and \eqref{sn:adm} give $\limsup_n w_n(2s_n)<\infty$. These two properties are enough for \eqref{problem:wn2}.
The proof is complete. 
\end{proof}

To satisfy \eqref{def:a(t)} under the setting of \eqref{setup}, we consider the sum of the right-hand side of \eqref{LwL:0}, with $t$ replaced by $\theta_n t$, and $\sup_{\xi\in S^{E_n}}\E^0_\xi\big[w_n\theta_n\overline{f_{\sigma_0} f_{\sigma_2\sigma_3}}(\xi_{\theta_n t})\big]$. Moreover, this supremum can be bounded by using \eqref{ineq:mutation} and $\P(M_{U_0,U_2}>\theta_n t)$, thanks to duality and the choice $\sigma_0\neq \sigma_2$. Therefore, given  $T\in(0,\infty)$, we set $a(t)=a_n(t)=\sum_{\ell=1}^3 a_{n,\ell}(t)$ for $t\in [0,T]$, where
\begin{align}\label{def:an(t)}
\begin{split}
a_{n,\ell}(t)&\defeq C_{\ref{def:an(t)}} \cdot 
 \frac{w_n\theta_n}{2\gamma_n\nu_n(\1)}\cdot w_n\theta_n\cdot 2\gamma_n\nu_n(\1)\int_0^T \P^{(n)}(M_{U_0,U_2}>\theta_n s)\d s\\
&\eqspace +C_{\ref{def:an(t)}}  \cdot  \frac{w_n\theta_n}{2\gamma_n\nu_n(\1)}\cdot 2\gamma_n\nu_n(\1)\P^{(n)}(M_{U_0,U_\ell}>\theta_n t) \\
&\eqspace +C_{\ref{def:an(t)}}  \cdot \frac{w_n\theta_n}{2\gamma_n\nu_n(\1)}\cdot (w_n\theta_n+1)\cdot 
\mu_n(\1)\theta_n\\
&\quad \quad \cdot 2\gamma_n\nu_n(\1)\int_0^T\P^{(n)}(M_{U_0,U_\ell}>\theta_ns)\d s
\end{split}
\end{align}
and $C_{\ref{def:an(t)}}$ depends only on $(\Pi,T)$. 
For any $n\geq 1$, $t\mapsto  a_{n}(t)$ is bounded and decreasing on $[0,T]$, and
\begin{align*}
\sup_{\xi\in S^{E_n}}\E^{w_n}_\xi\left[w_n\theta_n\overline{f_{\sigma_0} f_{\sigma_2\sigma_3}}(\xi_{\theta_nt})\right]&
\leq a_n(t),\quad \forall\;t\in [0,T].
\end{align*}
Hence, the conditions of $a_n(t)$ required in Proposition~\ref{prop:L2} hold.

For the proof of \eqref{closure}, the next step is to show that under the setting of \eqref{setup} and the above choice of $a(t)=a_n(t)$, the right-hand side of \eqref{ineq:L2} vanishes as $n\to\infty$. For the first term on the right-hand side of \eqref{ineq:L2}, proving $\int_0^{\delta_n}a_n(t)\d t$ amounts to proving $\int_0^{\delta_n}a_{n,\ell}(t)\d t\to 0$ for all $1\leq \ell\leq 3$. For the latter limits, note that $\delta_n\to 0$ by \eqref{sn:adm}. Also, a slight modification of the proofs of \eqref{problem:wn1}--\eqref{problem:wn2} shows that for the right-hand side of \eqref{def:an(t)}, the first and last terms in  are bounded in $n$, and the second term satisfies 
\begin{align*}
&\quad \lim_{n\to\infty}2\gamma_n\nu_n(\1)\int_0^{\delta_n}\P^{(n)}(M_{U_0,U_\ell}>\theta_n s)\d s\\
&=\lim_{n\to\infty}\frac{2s_n}{\theta_n}\cdot 2\gamma_n\nu_n(\1)\int_0^{1}\P^{(n)}(M_{U_0,U_\ell}>2s_n s)\d s=0.
\end{align*}
For the second term in \eqref{ineq:L2}, it is enough to show that $a_n(\delta_n)$'s are bounded. From the above argument for the first term in \eqref{ineq:L2}, this property follows if we use the second limit in \eqref{def:wn} and note that
\[
2\gamma_n\nu_n(\1)\P^{(n)}(M_{U_0,U_\ell}>\theta_n t)|_{t=\delta_n}=2\gamma_n\nu_n(\1)\P^{(n)}(M_{U_0,U_\ell}>2s_n )\xrightarrow[n\to\infty]{}  \overline{\kappa}_\ell,
\]
where the limit follows from Proposition~\ref{prop:sn-selection} and Proposition~\ref{prop:kell}. To use these propositions precisely, passing the foregoing limit actually requires that given any subsequence of $(E_n,q^{(n)})$, a suitable further subsequence is used. To lighten the exposition, we continue to suppress similar uses of subsequential limits.

For the third term in \eqref{ineq:L2}, note that $\delta_n\to 0$ by \eqref{sn:adm}, and the $g_n$'s  in \eqref{setup} are uniformly bounded in $n$. The last term in \eqref{ineq:L2} is the major term. By \eqref{setup} and Proposition~\ref{prop:duhamel} (2$\cc$), it remains to prove
\begin{align}\label{voter:density}
\lim_{n\to\infty}\sup_{\xi\in S^{E_n}}\big|\E^{0}_\xi[w_n\theta_n\overline{f_{\sigma_0} f_{\sigma_2\sigma_3}}(\xi_{2s_n})]-w_\infty Q_{\sigma_0,\sigma_2\sigma_3}\big(p(\xi)\big)\big|=0.
\end{align}
For the next lemma, recall that the total variation distance $d_E$ and the spectral gap $\mathbf g$ are  defined at the beginning of Section~\ref{sec:mainresults}.  Also, here and in what follows, we use the shorthand notation $\E[Z;A]=\E[Z\1_A]$.

\begin{lem}\label{lem:RW}
Fix $(\sigma_0,\sigma_2,\sigma_3)\in S\times S\times S$ such that $\sigma_0\neq \sigma_2$. \medskip 

\noindent {\rm (1$\cc$)}
Given any $0<s<t<\infty$, the following estimate of the voter model by the coalescing Markov chains holds:
\begin{align}
\begin{split}
&\quad  \sup_{\xi\in S^E}\Big|\E^0_\xi\big[\,\overline{f_{\sigma_0} f_{\sigma_2\sigma_3}}(\xi_t)\big]-
\1_{\{\sigma_2=\sigma_3\}}\P(M_{U_0,U_2}>s,M_{U_0,U_3}>s)p_{\sigma_0}(\xi)p_{\sigma_2}(\xi)\\
&\quad\quad+\1_{\{\sigma_2=\sigma_3\}}\P(M_{U_0,U_2}>s,M_{U_2,U_3}>s,M_{U_0,U_3}>s)p_{\sigma_0}(\xi)p_{\sigma_2}(\xi)\\
&\quad\quad  -\1_{\{\sigma_0=\sigma_3\}}\P(M_{U_0,U_2}>s,M_{U_2,U_3}>s)p_{\sigma_0}(\xi)p_{\sigma_2}(\xi)\\
&\quad\quad  +\1_{\{\sigma_0=\sigma_3\}}\P(M_{U_0,U_2}>s,M_{U_2,U_3}>s,M_{U_0,U_3}>s)p_{\sigma_0}(\xi)p_{\sigma_2}(\xi)\\
&\quad\quad  -\P(M_{U_0,U_2}>s,M_{U_2,U_3}>s,M_{U_0,U_3}>s)p_{\sigma_0}(\xi)p_{\sigma_2}(\xi)p_{\sigma_3}(\xi)
\Big|\\
&\leq C_{\ref{ineq:RW}}\sum_{\ell=1}^3\Gamma_\ell(s,t),\label{ineq:RW}
\end{split}
\end{align} 
where $C_{\ref{ineq:RW}}$ is a universal constant and
\begin{align}
\begin{split}\label{ineq:Well_TV}
\Gamma_\ell(s,t)&\defeq \P(M_{U_0,U_\ell}\in (s,t])\\
&\eqspace+\min\left\{\sqrt{\frac{\pi_{\max}}{\nu(\1)}}\e^{-\mathbf g(t-s)},\P(M_{U_0,U_\ell}>s)d_E(t-s)\right\}\\
&\eqspace + \big(1-\e^{-2\mu(\1)t}\big)\P(M_{U_0,U_\ell}>t)+
\mu(\1)\int_0^t\P(M_{U_0,U_\ell}>r)\d r.
\end{split}
\end{align}
{\rm (2$\cc$)}  The limit in \eqref{voter:density} holds. 
\end{lem}

\begin{proof}
(1$\cc$) First, we consider the case that there is no mutation. Roughly speaking, the method of this proof is to express $\E_\xi^0\big[\,\overline{f_{\sigma_0} f_{\sigma_2\sigma_3}}(\xi_t)\big]$ in terms of coalescing Markov chains before any two coalesce. This way we can express the coalescing Markov chains as independent Markov chains and compute the asymptotics of $\E_\xi^0\big[\,\overline{f_{\sigma_0} f_{\sigma_2\sigma_3}}(\xi_t)\big]$ by the $\ok$-constants defined in \eqref{def:kell}--\eqref{def:|||}. This idea goes back to \cite[Proposition~6.1]{CCC}.

Now, by duality and the assumption $\sigma_0\neq \sigma_2$, it holds that 
\begin{align}
&\quad\,\E_\xi^0\big[\,\overline{f_{\sigma_0} f_{\sigma_2\sigma_3}}(\xi_t)\big]\notag\\
&=\E[\1_{\sigma_0}\circ \xi(B^{U_0}_t)
\1_{\sigma_2}\circ \xi(B^{U_2}_t)\1_{\sigma_3}\circ \xi(B^{U_3}_t)]\notag\\
&=\1_{\{\sigma_2=\sigma_3\}}\E[\1_{\sigma_0}\circ \xi(B^{U_0}_t)\1_{\sigma_2}\circ \xi(B^{U_2}_t);M_{U_0,U_2}>t,M_{U_2,U_3}\leq t,M_{U_0,U_3}>t]\notag\\
&\quad +\1_{\{\sigma_0=\sigma_3\}}\E[\1_{\sigma_0}\circ \xi(B^{U_0}_t)\1_{\sigma_2}\circ \xi(B^{U_2}_t);M_{U_0,U_2}>t,M_{U_2,U_3}> t,M_{U_0,U_3}\leq t]\notag\\
&\quad +\E[\1_{\sigma_0}\circ \xi(B^{U_0}_t)\1_{\sigma_2}\circ \xi(B^{U_2}_t)\1_{\sigma_3}\circ \xi(B^{U_3}_t);M_{U_0,U_2}>t,M_{U_2,U_3}> t,M_{U_0,U_3}>t]\notag\\
&=\1_{\{\sigma_2=\sigma_3\}}{\rm I}+\1_{\{\sigma_0=\sigma_3\}}{\rm II}+{\rm III}.\label{def:III}
\end{align}
We can further write ${\rm I}$ and ${\rm II}$ as
\begin{align}
\begin{split}\label{def:Iterm}
{\rm I}&=\E[\1_{\sigma_0}\circ \xi(B^{U_0}_t)\1_{\sigma_2}\circ \xi(B^{U_2}_t);M_{U_0,U_2}>t,M_{U_0,U_3}>t]\\
&\quad \,-\E[\1_{\sigma_0}\circ \xi(B^{U_0}_t)\1_{\sigma_2}\circ \xi(B^{U_2}_t);M_{U_0,U_2}>t,M_{U_2,U_3}> t,M_{U_0,U_3}>t]\\
&={\rm I'}-{\rm I''},
\end{split}\\
\begin{split}\label{def:IIterm}
{\rm II}&=\E[\1_{\sigma_0}\circ \xi(B^{U_0}_t)\1_{\sigma_2}\circ \xi(B^{U_2}_t);M_{U_0,U_2}>t,M_{U_2,U_3}> t]\\
&\quad\, -\E[\1_{\sigma_0}\circ \xi(B^{U_0}_t)\1_{\sigma_2}\circ \xi(B^{U_2}_t);M_{U_0,U_2}>t,M_{U_2,U_3}> t,M_{U_0,U_3}> t]\\
&={\rm II'}-{\rm I''}.
\end{split}
\end{align}
We estimate ${\rm I'},{\rm I''}, {\rm II'}$ and ${\rm III}$ below, using the property that the coalescing Markov chains move independently before meeting.

First,  for $0<s<t<\infty$, it follows from the Markov property of independent $q$-Markov chains at time $s$ that ${\rm I}'$ in \eqref{def:Iterm} can be estimated as follows:
\begin{align}
\begin{split}\label{def:I1}
&\quad \big|{\rm I}'-\P(M_{U_0,U_2}>s,M_{U_0,U_3}>s)p_{\sigma_0}(\xi)p_{\sigma_2}(\xi)\big|\\
&\leq\E\Big[\left|\e^{(t-s)(q-1)}\1_{\sigma_0}\circ\xi(B_s^{U_0})\e^{(t-s)(q-1)}\1_{\sigma_2}\circ\xi(B_s^{U_2})-p_{\sigma_0}(\xi)p_{\sigma_2}(\xi)\right|\\
&\quad \quad ;M_{U_0,U_2}>s,M_{U_0,U_3}>s\Big]\\
&\eqspace+ \P(M_{U_0,U_2}\in (s,t])+\P(M_{U_0,U_3}\in (s,t]).
\end{split}
\end{align}
On the event $\{M_{U_0,U_2}>s,M_{U_0,U_3}>s\}$, $(B^{U_0}_r)_{0\leq r\leq s}$ and $(B^{U_2}_r)_{0\leq r\leq s}$ are independent $q$-Markov chains and each chain is stationary by the assumption on $\{U_\ell\}$. Since $p_\sigma(\xi)=\sum_x\1_\sigma\circ \xi(x)\pi(x)$, the expectation in \eqref{def:I1} can be estimate as in the proof of \cite[Proposition~6.1]{CCC}. We get
\begin{align}
&\eqspace\big|{\rm I}'-\P(M_{U_0,U_2}>s,M_{U_0,U_3}>s)p_{\sigma_0}(\xi)p_{\sigma_2}(\xi)\big|\notag\\
&\leq 
\min\left\{2\sqrt{\frac{\pi_{\max}}{\nu(\1)}}\e^{-\mathbf g(t-s)},4\P(M_{U_0,U_2}>s,M_{U_0,U_3}>s)d_E(t-s)\right\}\notag\\
&\eqspace+ \P(M_{U_0,U_2}\in (s,t])+\P(M_{U_0,U_3}\in (s,t]).\notag
\end{align}
Similar estimates apply to the other terms ${\rm I''}$, ${\rm II}'$ and ${\rm III}$ in \eqref{def:III}, \eqref{def:Iterm} and \eqref{def:IIterm}.

Applying all of these estimates to \eqref{def:III} proves \eqref{ineq:RW} when there is no mutation. The additional terms in \eqref{ineq:RW} arise when we include mutation and use \eqref{ineq:mutation} again.\medskip

\noindent (2$\cc$) Recall the second limit in \eqref{def:wn} and \eqref{problem:wn2}.
Then by \eqref{ineq:RW} and \eqref{ineq:Well_TV} with $t=2s_n$ and $s=s_n$, it suffices to show all of the following limits:
\begin{align}
&\lim_{n\to\infty} 2\gamma_n\nu_n(\1)\P^{(n)}(M_{U_0,U_\ell}\in (s_n,2s_n])=0,\quad 1\leq \ell\leq 3;\label{problem:wn4}\\
&\lim_{n\to\infty}\big(1-\e^{-2\mu_n(\1)\cdot (2s_n)}\big)\cdot 2\gamma_n\nu_n(\1)\P^{(n)}(M_{U_0,U_\ell}>2s_n),\quad 1\leq \ell\leq 3;\label{problem:wn5}\\
&\lim_{n\to\infty} \Gamma_{n,\ell}=0,\quad 1\leq \ell\leq 3,\label{problem:wn6}
\end{align}
where $\Gamma_{n,\ell}$ is given by the minimum of the following two terms:
\begin{align}\label{def:Gamma}
\gamma_n\nu_n(\1)\cdot \sqrt{\frac{\pi^{(n)}_{\max}}{\nu_n(\1)}}\e^{-\mathbf g_ns_n},\quad
 2\gamma_n\nu_n(\1)\P^{(n)}(M_{U_0,U_\ell}>s_n)d_{E_n}(s_n) .
\end{align}

To see \eqref{problem:wn4}, we simply use Propositions~\ref{prop:sn-selection} and \ref{prop:kell}. The limit in \eqref{problem:wn5} follows from the same propositions, in addition to \eqref{def:mun} and \eqref{sn:adm}. For \eqref{problem:wn6}, we consider the following two cases. When $\Gamma_n$ is given by the first term in \eqref{def:Gamma}, the required limit holds by \eqref{cond:pi} and the first limit in \eqref{cond2:sn}. When $\Gamma_n$ is given by  the other term in \eqref{def:Gamma}, we first use Propositions~\ref{prop:sn-selection} and \ref{prop:kell}. Then note that the second limit in \eqref{cond2:sn} implies $\lim_{n}\mathbf t^{(n)}_{\mix}/s_n=0$, and so, $\lim_nd_{E_n}(s_n)=0$ by \eqref{ineq:tmix}. We have proved \eqref{problem:wn6}. The proof is complete.
\end{proof}

Up to this point, we have proved the asymptotic closure of equation in the sense of \eqref{closure}. Note that under \eqref{setup}, the convergence of the last term in \eqref{ineq:L2} also contributes to asymptotic path regularity of the density processes. 

The next lemma proves the asymptotic path regularity more explicitly as tightness in the convergence results of Theorem~\ref{thm:main}. The limit of the normalized martingale terms in Theorem~\ref{thm:main} (2$\cc$) is also proven. Here, recall  that the density processes satisfy the  decompositions in \eqref{density:dynamics}. From now on,  $\xrightarrow[n\to\infty]{\rm (d)}$ refers to convergence in distribution as $n\to\infty$. 

\begin{lem}\label{lem:tight}
Fix $\sigma\in S$.  

\begin{enumerate}
\item [\rm (1$\cc$)] 
The sequence of laws of $I_\sigma^{(n)}$ as continuous processes under $\P^{w_n}_{\nu_n}$ is tight.

\item [\rm (2$\cc$)] 
The sequence of laws of $\1_{\{w_n>0\}}w_n^{-1}R_\sigma^{(n)}$ as continuous processes under $\P^{w_n}_{\nu_n}$  is tight.

\item [\rm (3$\cc$)] The sequence of laws of $M_\sigma^{(n)}$ as continuous processes under $\P^{w_n}_{\nu_n}$ converges to zero in distribution. 
\end{enumerate}
If, in addition, $\lim_n \gamma_n\nu_n(\1)/\theta_n=0$, then the following holds.
\begin{enumerate}
\item [\rm (4$\cc$)]  The sequence of laws of 
\begin{align}\label{def:product}
\left(\left(
\frac{\gamma_n}{\theta_n}\right)^{1/2}M_\sigma^{(n)}(t),\;\;
\frac{\gamma_n}{\theta_n}\langle M_\sigma^{(n)},M_\sigma^{(n)}\rangle_{t}-\int_0^t  p_\sigma(\xi_{\theta_ns})[1-p_{\sigma}(\xi_{\theta_ns})]\d s;t\geq 0\right)
\end{align}
as processes with c\`adl\`ag paths under $\P^{w_n}_{\nu_n}$ is $C$-tight, and the second coordinates 
converge to zero in distribution as processes. Moreover, for all $T\in(0,\infty)$,
\begin{align}\label{ineq:L2-bdd}
\sup_{n\geq 1}\sup_{t\in [0,T]}\sup_{\xi\in S^{E_n}}\frac{\gamma_n}{\theta_n}\E^{w_n}_{\xi}\big[M^{(n)}_\sigma(t)^2\big]<\infty.
\end{align}
\item [\rm (5$\cc$)]  For any $\sigma'\in S$ with $\sigma'\neq \sigma$, the sequence  
\begin{align}\label{def:product1}
\left(
\frac{\gamma_n}{\theta_n}\langle M_\sigma^{(n)},M_{\sigma'}^{(n)}\rangle_{t}+\int_0^t  p_\sigma(\xi_{\theta_ns})p_{\sigma'}(\xi_{\theta_ns})\d s;t\geq 0\right)
\end{align}
under $\P^{w_n}_{\nu_n}$ converges to zero in distribution as processes. 
\end{enumerate}
\end{lem}
\begin{proof}
(1$\cc$) First, we show a bound for $\sup_{\xi\in S^{E_n}}\E_{\xi}^{w_n}[|I^{(n)}_\sigma(\theta)|]$ explicitly in $\theta$. By \eqref{eq:Dsigma0},
\begin{align}\label{Dsigma:22}
|\overline{D}_\sigma(\xi)|\leq C_{\ref{Dsigma:22}}\sum_{\stackrel{\scriptstyle \sigma'\in S}{\sigma'\neq \sigma}}\overline{f_{\sigma\sigma'}}(\xi)
\end{align}
for some constant $C_{\ref{ineq:Itheta}}$ depending only on $\Pi$ and $\#S$. Indeed, for $q(x,y)>0$, $\1_\sigma\circ\xi(y)-\1_\sigma\circ\xi (x)\neq 0$ implies that either $\xi(x)$ or $\xi(y)$ is $\sigma$ but not both.  By  \eqref{def:In} and \eqref{Dsigma:22}, 
\begin{align}\label{In:continuity}
\begin{split}
\sup_{\xi\in S^{E_n}}\E_{\xi}^{w_n}[|I^{(n)}_\sigma(\theta)|]&\leq C_{\ref{Dsigma:22}}\sum_{\stackrel{\scriptstyle \sigma'\in S}{\sigma'\neq \sigma}}\int_0^\theta \sup_{\xi\in S^{E_n}}\E^{w_n}_{\xi}\left[w_n\theta_n\overline{f_{\sigma\sigma'}}(\xi_{\theta_n s})\right]\d s\\
&\quad +2\mu_n(\1)\theta_n\cdot \theta.
\end{split}
\end{align}
Furthermore, we can bound the expectations on the right-hand side of \eqref{In:continuity} by using an analogue  of the $a_n(t)$ in \eqref{def:an(t)}, but now involving only the meeting time $M_{V,V'}$. Specifically, given $T\in(0,\infty)$, the following inequality holds for all $t\in [0,T]$:
\begin{align}\label{def:an'(t)}
\begin{split}
&\eqspace\E^{w_n}_{\xi}\left[w_n\theta_n\overline{f_{\sigma\sigma'}}(\xi_{\theta_n t})\right]\\
&\leq  C_{\ref{def:an'(t)}} \cdot 
 \frac{w_n\theta_n}{2\gamma_n\nu_n(\1)}\cdot w_n\theta_n\cdot 2\gamma_n\nu_n(\1)\int_0^T \P^{(n)}(M_{V,V'}>\theta_n s)\d s\\
&\eqspace +C_{\ref{def:an'(t)}}  \cdot  \frac{w_n\theta_n}{2\gamma_n\nu_n(\1)}\cdot 2\gamma_n\nu_n(\1)\P^{(n)}(M_{V,V'}>\theta_n t) \\
&\eqspace +C_{\ref{def:an'(t)}}  \cdot \frac{w_n\theta_n}{2\gamma_n\nu_n(\1)}\cdot (w_n\theta_n+1)\cdot
\mu_n(\1)\theta_n\\
&\eqspace\quad \cdot  2\gamma_n\nu_n(\1)\int_0^T\P^{(n)}(M_{V,V'}>\theta_ns)\d s
\end{split}
\end{align}
and $C_{\ref{def:an'(t)}}$ depends only on $(\Pi,T,\sup_n\pi^{(n)}_{\max}/\pi^{(n)}_{\min})$. To see \eqref{def:an'(t)}, we combine \eqref{LwL:001} and \cite[Proposition~3.2]{CC} and then use \eqref{cond:pi} and \eqref{ineq:MUVcompare} to reduce probabilities of $M_{U_0,U_1}$ to probabilities of $M_{V,V'}$.

Next, we show that 
\begin{align}\label{ineq:Itheta}
\lim_{\theta\searrow  0}\limsup_{n\to\infty}\sup_{\xi\in S^{E_n}}\E_\xi^{w_n}[|I^{(n)}_\sigma(\theta)|]=0.
\end{align}
First, \eqref{def:mun} readily gives the required limit of the last term of \eqref{In:continuity}. We focus on the sum of integrals on the right-hand side of \eqref{def:mun}. For each of these integrals, note that the first and last terms in \eqref{def:an'(t)} are uniformly bounded in $n$ as in the case of \eqref{def:an(t)}. The integral over $t\in [0,\theta]$ of the second term in \eqref{def:an'(t)}  satisfies 
\begin{align}\label{MVV':conv-tight}
\lim_{\theta\searrow  0}\limsup_{n\to\infty}\frac{w_n\theta_n}{2\gamma_n\nu_n(\1)}\cdot \frac{2\gamma_n}{\theta_n}\nu_n(\1)\int_0^{\theta\theta_n} \P^{(n)}(M_{V,V'}> s)\d s=0
\end{align}
by the second limit of \eqref{def:wn}, \eqref{ergodic}, and \eqref{fnsn:est} with $s_n$ replaced by $\theta_n$ and with $t_2=\theta$ and $t_1=0$ since $(\theta_n)$ is also a slow sequence. (The use of $M_{V,V'}$ allows us to circumvent Lemma~\ref{lem:UVbdd} due to the explosion of the bound in \eqref{UVbdd:1} as $t\to 0$.) We have proved \eqref{ineq:Itheta}.

Finally, the required tightness follows from \eqref{ineq:Itheta}, the strong Markov property of the particle system and Aldous's criterion for tightness \cite[Proposition~VI.4.5 on p.356]{JS}. The detail is similar to the proof of \cite[Theorem~5.1 (1)]{CCC}.  \medskip

\noindent {\rm (2$\cc$)} Recall the equation \eqref{def:R} of $R_\sigma$, and the explicit form of $R^w$ can be read from \eqref{qwq:exp}. Then by the same reason for \eqref{Dsigma:22}, the coefficient of $R_\sigma$ satisfies 
\begin{align}\label{bdd:R}
\sum_{x,y\in E}\pi(x)|\1_\sigma\circ\xi(y)-\1_\sigma\circ\xi (x)|q(x,y)|R^w(x,y,\xi)|\leq C_{\ref{bdd:R}}\sum_{\sigma'\in S}\overline{f_{\sigma\sigma'}}(\xi),
\end{align}
where $C_{\ref{bdd:R}}$ depends only on $\Pi$ and $\#S$. From \eqref{bdd:R},
 the argument in (1$\cc$) applies again.  
\medskip

\noindent {\rm (3$\cc$)} The proof follows from a slight modification of the proof of (4$\cc$) below even without the additional assumption $\lim_n \gamma_n\nu_n(\1)/\theta_n=0$. \medskip 

\noindent {\rm (4$\cc$)}
We start with the convergence of the second coordinate in \eqref{def:product}. Define a density function $\widetilde{p}_{\sigma}(\xi)$ on $S^{E_n}$ such that the stationary weights $\pi^{(n)}(x)$ in $p_\sigma(\xi)$ are replaced by $\pi^{(n)}(x)^2/\nu_n(\1)$. From \eqref{def:<M1>}, the following equality holds under $\P^{w_n}_{\xi}$ for all $\xi\in S^{E_n}$:
\begin{align}\label{eq:Mn}
\begin{split}
&\quad \frac{\gamma_n}{\theta_n}\langle M_\sigma^{(n)},M_{\sigma}^{(n)}\rangle_t\\
&=\gamma_n\nu_n(\1)\int_0^t \sum_{\sigma'\in S\setminus\{\sigma\}}[p_{\sigma'\sigma}(\xi_{\theta_ns})+p_{\sigma\sigma'}(\xi_{\theta_ns})]\d s\\
&\quad +\gamma_n\nu_n(\1)\int_0^t \Big([1-\widetilde{p}_{\sigma}(\xi_{\theta_ns})]\mu_n(\sigma)+\widetilde{p}_\sigma(\xi_{\theta_ns})\mu_n(S\setminus\{\sigma\})\Big)\d s\\
&\quad + \frac{\gamma_n\nu_n(\1)}{\theta_n}\cdot w_n\theta_n\int_0^{t} \widetilde{R}^{(n)}_{w_n}(\xi_{\theta_n s})\d s.
\end{split}
\end{align}
Here, $\widetilde{R}^{(n)}_{w_n}$ can be bounded in the same way as \eqref{Dsigma:22}. Note that due to the use of $\widetilde{R}^{(n)}_{w_n}$, we only involve the first term $q(x,y)$ in  the expansion \eqref{qwq:exp} of $q^{w_n}(x,y,\xi)$. 

Let us explain how the required convergence of the second coordinate in \eqref{def:product} follows \eqref{eq:Mn}. First, since $\lim_n \gamma_n\nu_n(\1)/\theta_n=0$ by assumption, the bound for $\widetilde{R}^{(n)}_{w_n}$ mentioned above and the proof of (1$\cc$) show that the continuous process defined by the last integral of \eqref{eq:Mn} converges to zero in distribution. Next, for all $0\leq T_0<T_1<\infty$, it holds that 
\begin{align*}
&\eqspace\gamma_n\nu_n(\1)\int_{T_0}^{T_1} \Bigg(\sum_{\sigma'\in S\setminus\{\sigma\}}\widetilde{p}_{\sigma'}(\xi_{\theta_ns})\mu_n(\sigma)+\widetilde{p}_\sigma(\xi_{\theta_ns})\mu_n(S\setminus\{\sigma\})\Bigg)\d s\\
&\leq \frac{\gamma_n\nu_n(\1)}{\theta_n}\cdot \mu_n(\1)\theta_n\cdot (T_1-T_0)\xrightarrow[n\to\infty]{} 0
\end{align*}
by \eqref{def:mun} and the assumption $\lim_n \gamma_n\nu_n(\1)/\theta_n=0$. Hence, the second integral on the right-hand side of \eqref{eq:Mn} converges to zero in distribution as a continuous process. 

The sequence of laws of the first integrals on the right-hand side of \eqref{eq:Mn} is tight for a reason similar to \eqref{MVV':conv-tight}. Moreover, since  $\overline{\kappa}_0$ in \eqref{cond:kappa0} is $1$ by Proposition~\ref{prop:sn-selection}, a slight modification of the proof of \eqref{closure} shows that 
\[
\gamma_n\nu_n(\1)\int_0^t \sum_{\sigma'\in S\setminus\{\sigma\}}[p_{\sigma'\sigma}(\xi_{\theta_ns})+p_{\sigma\sigma'}(\xi_{\theta_ns})]\d s-\int_0^t  p_\sigma(\xi_{\theta_ns})[1-p_\sigma(\xi_{\theta_ns})]\d s\xrightarrow[n\to\infty]{\rm (d)} 0
\]
as processes. See \eqref{LwL:001} and \cite[Proposition~6.1]{CCC}.  By this convergence and the explanation in the preceding paragraph, \eqref{eq:Mn} thus shows that the sequence of laws of the second coordinates in \eqref{def:product}  as processes converges to zero in distribution. 

To get the $C$-tightness of the sequence of laws of the first coordinates in \eqref{cond:kappa0}, first, note that the tightness readily follows from the $C$-tightness of $(\gamma_n/\theta_n)\langle M^{(n)}_\sigma,M^{(n)}_\sigma\rangle$ proven above \cite[Theorem~VI.4.13 on p.358]{JS}. For the stronger $C$-tightness, note that the jump sizes of $p_\sigma(\xi_t)$ are given by $\pi(x)$'s. Hence,
\begin{align}\label{Mjump}
(\gamma_n/\theta_n)^{1/2}\sup_{t\geq 0}|\Delta M^{(n)}_\sigma(t)|\leq (\gamma_n/\theta_n)^{1/2}\pi^{(n)}_{\max},
\end{align}
where the right-hand side tends to zero by \eqref{cond:pi} and the assumption  $\lim_n \gamma_n\nu_n(\1)/\theta_n=0$. The required $C$-tightness now follows from \cite[Proposition~3.26 in Chapter VI on p.351]{JS}. 

Finally, the above argument for the tightness of $(\gamma_n/\theta_n)^{1/2} M_\sigma^{(n)}$ also proves \eqref{ineq:L2-bdd}. 
\medskip

\noindent {\rm (5$\cc$)} The proof is almost identical to the proof of (4$\cc$)  for the second coordinate of \eqref{def:product},  if we start with \eqref{def:<M2>}. The details are omitted. 
\end{proof}

\subsection{The replicator equation and the Wright--Fisher fluctuations}
In this subsection, we complete the proof of Theorem~\ref{thm:main} and give the proof of Corollary~\ref{cor:symmetric}.\\

\begin{proof}[Completion of the proof of Theorem~\ref{thm:main}]
By \eqref{eq:Dsigma}, \eqref{def:In}, \eqref{closure} and Lemma~\ref{lem:tight} (1$\cc$)--(3$\cc$), we have proved that the following  vector process converges to zero in distribution:
\begin{align*}
&p_\sigma(\xi_{\theta_n t})-\int_0^tw_\infty\Bigg(\sum_{\stackrel{\scriptstyle \sigma_0,\sigma_3\in S}{ \sigma_0\neq\sigma}}\Pi(\sigma,\sigma_3) Q_{\sigma_0,\sigma\sigma_3}\big(p(\xi_{\theta_n s})\big)-\sum_{\stackrel{\scriptstyle \sigma_2,\sigma_3\in S}{\sigma_2\neq\sigma}}\Pi(\sigma_2,\sigma_3) Q_{\sigma,\sigma_2\sigma_3}\big(p(\xi_{\theta_n s})\big)\Bigg)\d s\\
&\eqspace-\int_0^t \Bigg(\mu_\infty(\sigma)[1-p_{\sigma}(\xi_{\theta_ns})]-\mu_\infty(S\setminus\{\sigma\}) p_\sigma(\xi_{\theta_ns})\Bigg)\d s,\quad \sigma\in S,
\end{align*}
where the polynomials $Q_{\sigma_0,\sigma_2\sigma_3}$ are defined in \eqref{def:Qsigma}. Hence, the sequence of laws of $p(\xi_{\theta_nt})$ is $C$-tight, and $p(\xi_{\theta_nt})$ converges in distribution to $X(t)$ as processes, where $X$ is the unique solution to the following system:
\begin{align}\label{p1:lim0}
\dot{X}_\sigma=w_\infty Q_\sigma(X)+\mu_\infty(\sigma)(1-X_\sigma)-\mu_\infty(S\setminus\{\sigma\}) X_\sigma,\quad \sigma\in S,
\end{align}
and the polynomial $Q_\sigma(X)$ in \eqref{p1:lim0} is given by
\begin{align}\label{def:grandQ}
Q_\sigma(X)&=\sum_{\stackrel{\scriptstyle \sigma_0,\sigma_3\in S}{\sigma_0\neq \sigma}}\Pi(\sigma,\sigma_3)Q_{\sigma_0,\sigma\sigma_3}(X)-\sum_{ \stackrel{\scriptstyle \sigma_2,\sigma_3\in S}{\sigma_2\neq \sigma}}\Pi(\sigma_2,\sigma_3)Q_{\sigma,\sigma_2\sigma_3}(X).
\end{align}

To simplify \eqref{def:grandQ} to the required form in \eqref{p1:lim}, note that the constraints $\sigma_0\neq \sigma$ and $\sigma_2\neq \sigma$ in \eqref{def:grandQ} can be removed from the definition of $Q_\sigma(X)$ by cancelling repeating terms. In doing so, we extend the definition $Q_{\sigma_0,\sigma_2\sigma_3}(X)$ to $\sigma_0=\sigma_2$ by the same formula in  \eqref{def:Qsigma}, but only in this proof. We also  lighten notation by the following: $A=\overline{\kappa}_{(2,3)|0}-\overline{\kappa}_{0|2|3}$ and $B=\overline{\kappa}_{(0,3)|2}-\overline{\kappa}_{0|2|3}$ and $C=\overline{\kappa}_{0|2|3}$. Then by \eqref{def:grandQ},
\begin{align*}
&Q_\sigma(X)=\sum_{\sigma_0,\sigma_3\in S}\Pi(\sigma,\sigma_3)\left.\left(\1_{\{\sigma_2=\sigma_3\}}AX_{\sigma_0}X_{\sigma_2}+\1_{\{\sigma_0=\sigma_3\}}BX_{\sigma_0}X_{\sigma_2}+CX_{\sigma_0}X_{\sigma_2}X_{\sigma_3}\right)\right|_{\sigma_2=\sigma}\\
&\quad \;-\sum_{\sigma_2,\sigma_3\in S}\Pi(\sigma_2,\sigma_3)\left.\left(\1_{\{\sigma_2=\sigma_3\}}AX_{\sigma_0}X_{\sigma_2}+\1_{\{\sigma_0=\sigma_3\}}BX_{\sigma_0}X_{\sigma_2}+CX_{\sigma_0}X_{\sigma_2}X_{\sigma_3}\right)\right|_{\sigma_0=\sigma}\\
&=X_\sigma \sum_{\sigma_0\in S}A\Pi(\sigma,\sigma)X_{\sigma_0}+X_\sigma\sum_{\sigma_0\in S}B
\Pi(\sigma,\sigma_0)X_{\sigma_0}+X_\sigma \sum_{\sigma_3\in S}C\Pi(\sigma,\sigma_3)X_{\sigma_3}\\
&\quad \;-X_\sigma\sum_{\sigma_2\in S}A\Pi(\sigma_2,\sigma_2)X_{\sigma_2}-X_\sigma \sum_{\sigma_2\in S}B\Pi(\sigma_2,\sigma)X_{\sigma_2}\\
&\quad -
X_\sigma\sum_{\sigma_2\in S}\Bigg(\sum_{\sigma_3\in S}C\Pi(\sigma_2,\sigma_3)X_{\sigma_3}\Bigg)X_{\sigma_2}\\
&=X_\sigma\Bigg(A\Pi(\sigma,\sigma)+\sum_{\sigma'\in S}B
[\Pi(\sigma,\sigma')-\Pi(\sigma',\sigma)]X_{\sigma'}+\sum_{\sigma'\in S}C\Pi(\sigma,\sigma')X_{\sigma'}\Bigg)\\
&\quad -X_\sigma \sum_{\sigma'\in S}\Bigg(A\Pi(\sigma',\sigma')+\sum_{\sigma''\in S}C\Pi(\sigma',\sigma'')X_{\sigma''}\Bigg)X_{\sigma'}.
\end{align*}
Note that we have used the property $\sum_{\sigma}X_{\sigma}=1$ in the last two equalities. The last equality is enough for the required form in \eqref{p1:lim} upon recalling \eqref{p1:lim0} and involving the polynomials $F_\sigma(X)$ and $\widetilde{F}_\sigma(X)$ in \eqref{F1} and \eqref{F2}. Moreover, \eqref{kappa:>} holds by Proposition~\ref{prop:sn-selection}, \eqref{ineq:MUVcompare}, and Proposition~\ref{prop:kell} (1$\cc$) and (4$\cc$). 

For the proof of (2$\cc$), notice that  by (1$\cc$) and Lemma~\ref{lem:tight} (4$\cc$)--(5$\cc$), the following convergence of matrix processes holds:
\[
\left(\frac{\gamma_n}{\theta_n}\langle M_\sigma^{(n)},M_{\sigma'}^{(n)}\rangle_t\right)_{\sigma,\sigma'\in S}\xrightarrow[n\to\infty]{\rm (d)}\left(\int_0^t X_\sigma(s)[\delta_{\sigma,\sigma'}-X_{\sigma'}(s)]\d s\right)_{\sigma,\sigma'\in S}.
\]
By this convergence and \eqref{ineq:L2-bdd}, the standard martingale problem argument shows that every weakly convergent subsequence of $((\gamma_n/\theta_n)^{1/2} M_\sigma^{(n)};\sigma\in S)$ converges to a continuous vector $L_2$-martingale $(M_\sigma^{(\infty)};\sigma\in S)$ with a quadratic variation matrix given by $(\int_0^t X_\sigma(s)[\delta_{\sigma,\sigma'}-X_{\sigma'}(s)]\d s;\sigma,\sigma'\in S)$. See \cite[Proposition~1.12 in Chapter~IX on p.525]{JS} and the proof of \cite[Theorem~1.10 in Chapter~XIII on pp.519--520]{RY}. Hence, the limiting vector martingale $(M_\sigma^{(\infty)};\sigma\in S)$ is a Gaussian process with covariance matrix $(\int_0^t X_\sigma(s)[\delta_{\sigma,\sigma'}-X_{\sigma'}(s)]\d s;\sigma,\sigma'\in S)$ \cite[Exercise (1.14) in Chapter~V on p.186]{RY}. Moreover, by uniqueness in law of this Gaussian process, the convergence holds along the whole sequence of the vector martingale $(\gamma_n/\theta_n)^{1/2}M^{(n)}$. The proof is complete.
\end{proof}

\begin{proof}[Proof of Corollary~\ref{cor:symmetric}]
Write $u$ for $X_1$. In this case, $X_0=1-u$ and the polynomial $Q_1(X)$ defined by \eqref{def:grandQ} simplifies to
\[
Q_1=(b-c)Q_{0,11}-cQ_{0,10}-bQ_{1,01}=(Q_{0,11}-Q_{1,01})b-(Q_{0,11}+Q_{0,10})c.
\]
By \eqref{def:Qsigma}, the coefficient of $c$ is given by 
\begin{align}\label{coeff:c}
\begin{split}
-Q_{0,11}(X)-Q_{0,10}(X)&=-(\overline{\kappa}_{(2,3)|0}-\overline{\kappa}_{0|2|3})(1-u)u-\overline{\kappa}_{0|2|3}(1-u)u^2\\
&\quad -(\overline{\kappa}_{(0,3)|2}-\overline{\kappa}_{0|2|3})(1-u)u-\overline{\kappa}_{0|2|3}(1-u)^2u,
\end{split}
\end{align}
and the coefficient of $b$ is
\begin{align}\label{coeff:b}
\begin{split}
Q_{0,11}(X)-Q_{1,01}(X)&=(\ok_{(2,3)|0}-\ok_{0|2|3})(1-u)u+\overline{\kappa}_{0|2|3}u^2(1-u)u^2\\
&\quad -(\overline{\kappa}_{(0,3)|2}-\overline{\kappa}_{0|2|3})(1-u)u-\overline{\kappa}_{0|2|3}(1-u)u^2.
\end{split}
\end{align}

These two coefficients can be simplified by using the definition of $Q_{\sigma_0,\sigma_2\sigma_3}$ and Proposition~\ref{prop:kell} (4$\cc$), if we follow the algebra in the proof of Lemma~\ref{lem:D} that simplifies \eqref{eq:Dsigma} to \eqref{eq:Dsigma1}. For example, a similar argument as in the proof of Lemma~\ref{lem:RW} shows that \eqref{voter:density} holds with 
\[
Q_{0,01}(X)=(\overline{\kappa}_{(0,2)|3}-\overline{\kappa}_{0|2|3})(1-u)u+\overline{\kappa}_{0|2|3}(1-u)^2u,
\]
and so
\begin{align*}
Q_{0,11}(X)+Q_{0,01}(X)&=(\overline{\kappa}_{(2,3)|0}-\overline{\kappa}_{0|2|3})(1-u)u+\overline{\kappa}_{0|2|3}(1-u)u^2\\
&\eqspace +(\overline{\kappa}_{(0,2)|3}-\overline{\kappa}_{0|2|3})(1-u)u+\overline{\kappa}_{0|2|3}(1-u)^2u \\
&=(1-u)u\ok_{3},
\end{align*}
where the last equality follows from Proposition~\ref{prop:kell} (4$\cc$). In this way, we can obtain from \eqref{coeff:c} and \eqref{coeff:b} that $Q_1(X)=[(\ok_3-\ok_1)b-\ok_2 c](1-u)u$. Moreover, by Proposition~\ref{prop:kell}, we can pass limit along the whole sequence to get this limiting polynomial $Q_1(X)$.
 \end{proof}

\section{Further properties of coalescing lineage distributions}\label{sec:coal}

\subsection{A comparison with mutations}
In this section, we prove some auxiliary results for the proof of Theorem~\ref{thm:main}. The next proposition estimates the voter model $(\xi_t)$ under $\P^0$ by its selection mechanism, that is, by the updates from $\{\Lambda(x,y);x,y\in E\}$. The proof extends \cite[Proposition~3.2]{CC}. Recall the notation  in Section~\ref{sec:dynamics} for the coalescing Markov chains. 

\begin{prop}\label{prop:mutation}
{\rm (1$\cc$)} Let $f:S\times S\times S\to [-1,1]$ be a function such that $f(\sigma,\sigma,\cdot)=0$ for all $\sigma\in S$. 
Then for all $t\in (0,\infty)$ and $x,y,z\in E$,
\begin{align}\label{ineq:mutation}
\begin{split}
& \sup_{\xi\in S^E}\Big|\E^0_\xi\big[f\big(\xi_t(x),\xi_t(y),\xi_t(z)\big)\big]-\E\big[f\big(\xi(B^x_t),\xi(B^y_t),\xi(B^z_t)\big)\big]\Big|\\
&\quad \leq  \big(1-\e^{-2\mu(\1)t}\big)\P(M_{x,y}>t)+\1_{x\neq y}\big(1-\e^{-\mu(\1)t}\big)\P(M_{x,z}\wedge M_{y,z}>t)\\
&\quad \quad +2\mu(\1)\int_0^t \P(M_{x,y}>s)\d s+\1_{x\neq y}\mu(\1)\int_0^t \P(M_{x,z}\wedge M_{y,z}>s)\d s.
\end{split}
\end{align}
{\rm (2$\cc$)} For all $(\sigma_0,\sigma_2,\sigma_3)\in S\times S\times S$ with $\sigma_0\neq \sigma_2$, $t\in(0,\infty)$ and $x\in E$,
\begin{align}\label{claim:w0-mutation}
\begin{split}
&\eqspace\sup_{\xi\in S^E}\big|\E^0_{\xi^x}\left[\,\overline{f_{\sigma_0} f_{\sigma_2\sigma_3}}(\xi_t)\right]-\E^0_\xi\left[\,\overline{f_{\sigma_0} f_{\sigma_2\sigma_3}}(\xi_t)\right]\big|\\
&\leq \sum_{\ell\in \{0,2,3\}}4 \P(M_{U_0,U_2}>t,B^{U_\ell}_t=x)\\
&\quad +\sum_{\ell\in \{0,2,3\}} 4\mu(\1)\int_0^t \P(M_{U_0,U_2}>s,B^{U_\ell}_t=x)\d s.
\end{split}
\end{align}
\noindent {\rm (3$\cc$)} For all $(\sigma_0,\sigma_1)\in S\times S$ with $\sigma_0\neq \sigma_1$, $t\in(0,\infty)$ and $x\in E$,
\begin{align}\label{claim:w0-mutation1}
\begin{split}
&\eqspace\sup_{\xi\in S^E}\big|\E^0_{\xi^x}\left[\,\overline{f_{\sigma_0\sigma_1}}(\xi_t)\right]-\E^0_\xi\left[\,\overline{ f_{\sigma_0\sigma_1}}(\xi_t)\right]\big|\\
&\leq \sum_{\ell\in \{0,1\}}4 \P(M_{U_0,U_1}>t,B^{U_\ell}_t=x)\\
&\quad +\sum_{\ell\in \{0,1\}} 4\mu(\1)\int_0^t \P(M_{U_0,U_1}>s,B^{U_\ell}_t=x)\d s.
\end{split}
\end{align}
\end{prop}
\mbox{}

The proof of this proposition extends the proof of \cite[Proposition~3.2]{CC} and is based on the pathwise duality between the voter model and the coalescing Markov chains. The relation follows from time reversal of the stochastic integral equations in Section~\ref{sec:mainresults} of the voter model. More specifically, for fixed $t\in (0,\infty)$, we define  a system of coalescing $q$-Markov chains $\{B^{a,t};a\in E\}$ such that in the absence of mutation, $B^{a,t}$ traces out the time-reversed ancestral line that determines the type at $(a,t)$ under the voter model. For example, if $s$ is the last jump time of $\{\Lambda_r(a,b);b\in E,r\in (0,t]\}$ and $\Lambda(a,c)$ causes this jump, the state of $B^{a,t}$ stays at $a$ before transitioning to $B^{a,t}_{t-s}=c$. Similarly, with the Poisson processes $\Lambda^\sigma$ driving the mutations, we can define $e(a,t)$ and $M(a,t)$ for the time and the type from the first mutation event on the trajectory of $B^{a,t}$, with $e(a,t)=\infty$ if there is no mutation. Since $e(a,t)>t$ if and only if $e(a,t)=\infty$, we have
\begin{align}\label{prob:dual}
\xi_t(a)=M(a,t)\1_{\{e(a,t)\leq t\}}+\xi\big(B^{a,t}_t\big)\1_{\{e(a,t)>t\}},\quad\forall\;a\in E,\quad\mbox{$\P^0_\xi$-a.s.}
\end{align}
More details can be seen by modifying the description in \cite[Section~6.1]{CC}. In the absence of mutation, this relation between the duality and the stochastic integral equations is  known in \cite{MT}.

We also observe two identities for the probability distributions of  the mutation times $e(a,t)$'s when we condition on $\G\defeq\sigma(\Lambda(a,b);a,b\in E)$. Let $x,y\in E$. Write $0=J_0<J_1<\cdots<J_N<J_{N+1}=  t$ such that $J_1,J_2,\cdots,J_N$ are the jump times of the bivariate chain $(B^{x,t},B^{y,t})$. Hence, $B^{x,t}_r=x_k$ and $B_r^{y,t}=y_k$, for all  $r\in [J_{k},J_{k+1})$ and $0\leq k\leq N$. First, fix $s\in [0,\infty)$, and note that conditioned on $\G$, whether mutation occurs along the trajectory of $B^{x,t}$ over $[J_k,s\wedge J_{k+1})$ depends on whether $\sum_{\sigma\in S}[\Lambda^\sigma_{(t-J_k)-}(x_k)-\Lambda_{(t-s\wedge J_{k+1})-}^\sigma(x_k)]\geq 1$. Note that $s'\mapsto \sum_{\sigma\in S}[\Lambda^\sigma_t(x_k)-\Lambda^\sigma_{(t-s')-}(x_k)]$ is a Poisson process with rate $\mu(\1)$. Hence, summing over $\sum_{\sigma\in S}[\Lambda^\sigma_{(t-J_k)-}(x_k)-\Lambda_{(t-s\wedge J_{k+1})-}^\sigma(x_k)]$ in $k$, we deduce
\begin{align}\label{eq:e1}
 \P\big(e(x,t)\leq s|\G\big)=1-\e^{-\mu(\1)s},\quad s\geq 0.
\end{align}
Second, note that $x_k\neq y_k$ for all $k$ such that $J_{k+1}\leq M_{x,y}\wedge  s$. In this case, mutations in $[J_{k},J_{k+1})$ along the trajectories of $B^{x,t}$ and $B^{y,t}$ are determined by two disjoint subsets of the Poisson processes $\{\Lambda^\sigma(a);\sigma\in S,a\in E\}$. Hence, \eqref{eq:e1} generalizes to the following identity: 
\begin{align}
 \P\big(e(x,t)\wedge e(y,t)\leq s\wedge M_{x,y}|\G\big)= 1-\e^{-2\mu(\1)(s\wedge M_{x,y})},\quad s\geq 0.\label{eq:e2}
\end{align}

\begin{proof}[\bf Proof of Proposition~\ref{prop:mutation}]
\noindent {\rm (1$\cc$)} Set a partition $\{A_j\}_{1\leq j\leq 4}$ as follows:
\begin{align}
\begin{split}\label{A1-A4}
A_1&=\{e(x,t)\wedge e(y,t)\leq t<M_{x,y}\},\\
A_2&=\{e(x,t)\wedge e(y,t)\leq M_{x,y}\leq t\},\\
A_3&=\{M_{x,y}<e(x,t)\wedge e(y,t)\leq t\},\\
A_4&=\{e(x,t)\wedge e(y,t)>t\}.
\end{split}
\end{align}
Then consider the corresponding differences for the left-hand side of \eqref{ineq:mutation}:
\begin{align}\label{def:Deltaj}
\Delta_j=\E^0_\xi\big[f\big(\xi_t(x),\xi_t(y),\xi_t(z)\big);A_j\big]-\E\big[f\big(\xi(B^x_t),\xi(B^y_t),\xi(B^z_t)\big);A_j\big],\quad 1\leq j\leq 4.
\end{align}
Let $\mathbf e_1$ and $\mathbf e_2$ be i.i.d. exponential random variables with mean $1/\mu(\1)$. It follows from \eqref{eq:e2} and the independence between selection and mutation that
\begin{align}\label{ineq:Delta1}
|\Delta_1|&\leq \P(\mathbf e_1\wedge \mathbf e_2\leq t)\P(M_{x,y}>t)= \big(1-\e^{-2\mu(\1)t}\big)\P(M_{x,y}>t),\\
\label{ineq:Delta2}
|\Delta_2|&\leq\int_0^t \P(t\geq M_{x,y}>s)\P(\mathbf e_1\wedge \mathbf e_2\in \d s)\leq 2\mu(\1)\int_0^t \P(M_{x,y}>s)\d s.
\end{align}
On $A_3$, $B^{x,t}_t=B^{y,t}_t$ by coalescence, and hence, $\xi_t(x)=\xi_t(y)$ by \eqref{prob:dual}. It follows from the assumption on $f$ that both of the expectations defining $\Delta_3$ are zero. 

To bound $\Delta_4$,  fix $z\in E$ and partition $A_4$ into the following four sets:
\begin{align*}
A_{41}&=\{e(x,t)\wedge e(y,t)>t, e(z,t)\leq t<M_{x,z}\wedge M_{y,z}\},\\
A_{42}&=\{e(x,t)\wedge e(y,t)>t, e(z,t)\leq M_{x,z}\wedge M_{y,z}\leq t\},\\
A_{43}&=\{e(x,t)\wedge e(y,t)>t,M_{x,z}\wedge M_{y,z}< e(z,t)\leq t\},\\
A_{44}&=\{e(x,t)\wedge e(y,t)>t,e(z,t)>t\}.
\end{align*}
Then define $\Delta_{4k}$ for $1\leq k\leq 4$ as in \eqref{def:Deltaj} by replacing $A_j$ with $A_{4k}$. By \eqref{eq:e1} and similar arguments for \eqref{ineq:Delta1} and \eqref{ineq:Delta2}, we get
\begin{align}\label{ineq:Delta412}
\begin{split}
|\Delta_{41}|&\leq \1_{x\neq y}\big(1-\e^{-\mu(\1)t}\big)\P(M_{x,z}\wedge M_{y,z}>t),\\
|\Delta_{42}|&\leq \1_{x\neq y}\mu(\1)\int_0^t \P(M_{x,z}\wedge M_{y,z}>s)\d s,
\end{split}
\end{align}
where the use of the indicator function $\1_{x\neq y}$ follows from the assumption of $f$. For $\Delta_{43}$, it is zero because $A_{43}=\varnothing$. Indeed, on  $\{M_{x,z}\wedge M_{y,z}< e(z,t)\leq t\}$,  either $e(x,t)\leq t$ or $e(y,t)\leq t$ since either $e(x,t)=e(z,t)$ (if $M_{x,z}\wedge M_{y,z}=M_{x,z}$) or $e(y,t)=e(z,t)$ (if $M_{x,z}\wedge M_{y,z}=M_{y,z}$). Hence, $\{M_{x,z}\wedge M_{y,z}< e(z,t)\leq t\}$ does not intersect $\{e(x,t)\wedge e(y,t)>t\}$. Finally, $\Delta_{44}=0$ by \eqref{prob:dual} now that the random variables being taken expectation are actually equal. 

In summary, we have proved that $\Delta_3=\Delta_{43}=\Delta_{44}=0$. In addition, $\Delta_{1}$, $\Delta_2$, $\Delta_{41}$ and $\Delta_{42}$ satisfy \eqref{ineq:Delta1}, \eqref{ineq:Delta2} and \eqref{ineq:Delta412}. We have proved \eqref{ineq:mutation}.\medskip

\noindent (2$\cc$) For the left-hand side of \eqref{claim:w0-mutation}, we use \eqref{prob:dual} to write
\begin{align}
&\eqspace\E^0_{\xi^x}\left[\,\overline{f_{\sigma_0} f_{\sigma_2\sigma_3}}(\xi_t)\right]-\E^0_\xi\left[\,\overline{f_{\sigma_0} f_{\sigma_2\sigma_3}}(\xi_t)\right]\notag\\
\begin{split}\label{A1-A4:0000}
&=\E\Bigg[\prod_{j\in \{0,2,3\}}\1_{\sigma_j}\Big(M(U_j,t)\1_{\{e(U_j,t)\leq t\}}+\xi^x(B^{U_j,t}_{t})\1_{\{e(U_j,t)> t\}}\Big)\\
&\eqspace-\prod_{j\in \{0,2,3\}}\1_{\sigma_j}\Big(M(U_j,t)\1_{\{e(U_j,t)\leq t\}}+\xi(B^{U_j,t}_{t})\1_{\{e(U_j,t)> t\}}\Big)\Bigg].
\end{split}
\end{align}
Mutation neglects the role of the initial condition. 
Hence, to get a nonzero value for the difference inside the foregoing expectation, we cannot have $e(U_j,t)\leq t$ for all $j\in \{0,2,3\}$.
 In this case, at least one of the sums
$\1_{\sigma_j}\circ \xi(B^{U_j,t}_t)+\1_{\sigma_j}\circ \xi^x(B^{U_j,t}_t)$, $j\in \{0,2,3\}$, has to be nonzero. We must have $B^{U_j,t}_t=x$ for some $j\in \{0,2,3\}$. By bounding the indicator functions associated with $\sigma_3$ by $1$, we obtain from \eqref{A1-A4:0000} that 
\begin{align}
&\eqspace\big|\E^0_{\xi^x}\left[\,\overline{f_{\sigma_0} f_{\sigma_2\sigma_3}}(\xi_t)\right]-\E^0_\xi\left[\,\overline{f_{\sigma_0} f_{\sigma_2\sigma_3}}(\xi_t)\right]\big|\notag\\
&\leq \sum_{\ell\in \{0,2,3\}}\left(\E_{\xi^x}+\E_{\xi}\right)\Bigg[\prod_{j\in \{0,2\}}\1_{\sigma_j}\circ \xi_t(U_j);B^{U_\ell,t}_t=x\Bigg],\quad\forall\;x\in E.\label{A1-A4:main1}
\end{align}

The method in (1$\cc$) now enters to remove mutations in each of the two expectations in the $\ell$-th summand of \eqref{A1-A4:main1}. For $\eta\in S^E$, we consider
\begin{align}\label{A1-A4-diff}
\E_\eta\Bigg[\prod_{j\in \{0,2\}}\1_{\sigma_j}\circ \xi_t(U_j);B^{U_\ell,t}_t=x\Bigg]-\E\Bigg[\prod_{j\in \{0,2\}}\1_{\sigma_j}\circ \eta(B^{U_j,t}_t);B^{U_\ell,t}_t=x\Bigg]
\end{align}
and use only the partition in \eqref{A1-A4} with $x=U_0$ and $y=U_2$. In this case, on $A_4$, the two products of the indicator functions in \eqref{A1-A4-diff} are equal. Since $\sigma_0\neq \sigma_2$ ensures that the second expectation in \eqref{A1-A4-diff} can be bounded by $\P(M_{U_0,U_2}>t,B^{U_\ell}_t=x)$, \eqref{A1-A4-diff} and a slight extension of \eqref{ineq:Delta1} and \eqref{ineq:Delta2} give
\begin{align}
&\eqspace\E_\eta\Bigg[\prod_{j\in \{0,2\}}\1_{\sigma_j}\circ \xi_t(U_j);B^{U_\ell,t}_t=x\Bigg]\notag\\
&\leq \P(M_{U_0,U_2}>t,B^{U_\ell}_t=x)+\big(1-\e^{-2\mu(\1)t}\big)\P(M_{U_0,U_2}>t,B^{U_\ell}_t=x)\notag\\
&\eqspace+2\mu(\1)\int_0^t \P(M_{U_0,U_2}>s,B^{U_\ell}_t=x)\d s,\quad \forall\;\eta\in S^E,\;x\in E.\label{A1-A4:main2}
\end{align}
The required inequality \eqref{claim:w0-mutation} now follows from \eqref{A1-A4:main1} and \eqref{A1-A4:main2}.\medskip 

\noindent {\rm (3$\cc$)} The proof of \eqref{claim:w0-mutation1} is almost the same as the proof of \eqref{claim:w0-mutation}  and is omitted. 
\end{proof}

\subsection{Full decorrelation on large random regular graphs}\label{sec:rrg}
In this subsection, we give a different proof of the explicit form of \eqref{cond:kappa0} by using the graphs' local convergence. Throughout the rest of this subsection, we use the graph-theoretic  terminologies from \cite{Bollobas,Chung}.

We start with the definition of the random regular graphs.  Fix an integer $k\geq 3$. Choose a sequence $\{N_n\}$ of positive integers such that $N_n\to\infty$ and $k$-regular graphs (without loops and multiple edges) on $N_n$ vertices exist. The existence of $\{N_n\}$ follows from the Erd\H{o}s--Gallai necessary and sufficient condition. Then the random $k$-regular graph on $N_n$ vertices is the graph $G_n$ chosen uniformly  from the set of $k$-regular graphs with $N_n$ vertices. We assume that the randomness defining the graphs is collectively subject to the probability $\mathbf P$ and the expectation $\mathbf E$. 

For applications to the evolutionary dynamics, we need two properties of random walks on the random graphs. See \cite[Section~3]{C:MT} and the references there for more details. First, the random walks are asymptotically irreducible in the following sense:
 \begin{align}\label{prob:comp}
\mathbf P(G_n\mbox{ has only one connected component})\to 1\quad\mbox{ as $n\to\infty$.}
\end{align} 
This property follows since the $\mathbf P$-probability that $G_n$ has a nonzero spectral gap tends to one \cite{Friedman, Bordenave}. See \cite[Lemma~1.7 (d) on pp.6--7]{Chung} for connections between graph spectral gaps and numbers of connected components. Second, $G_n$ for large $n$ is locally like the infinite $k$-regular tree $G_\infty$ in the following sense. Write $q^{(n),\ell}(x,y)$ for the $\ell$-step transition probability of random walk on $G_n$. For any $n, r\in \Bbb N$, write $\mathcal T_n(r)$ for the set of vertices $x$ in $G_n$ such that the subgraph induced by vertices $y$ with $d(x,y)<r$ does not have a cycle, where $d$ denotes the graph distance on $G_n$. Then a standard result for the random graphs $\{G_n\}$ \cite[Section~2.4]{Bollobas:RG} gives 
\begin{align}\label{NTn}
\frac{N_n-\# \mathcal T_n(\ell)}{N_n}\xrightarrow[n\to\infty]{\mathbf P} 0,\quad\forall\;\ell\geq 1,
\end{align}
where $\xrightarrow[n\to\infty]{\mathbf P}$ refers to convergence in $\mathbf P$-probability. Note that $\pi^{(n)}$ is uniform on $G_n$. Consequently, if $q^{(\infty),\ell}$ stands for the $\ell$-step transition probability of random walk on $G_\infty$, then \eqref{NTn} implies that for all $L\in \Bbb N$,
\begin{align}\label{LTL}
\pi^{(n)}\big\{x\in \mathcal T_{n}(2L);q^{(n),\ell}(x,y)=q^{(\infty),\ell}(x,y),\forall\;y,\;\forall\;\ell\in \{1,2,\cdots,L\}\big\}\xrightarrow[n\to\infty]{\mathbf P}1.
\end{align}
Below we write $\P^{(n)}$ and $\E^{(n)}$ for the random walk probability and the expectation under $q^{(n)}=q^{(n),1}$ for $n\in \Bbb N\cup \{\infty\}$. Notations for meeting times, random walks, and related objects associated with $q^{(n)}$ extend to $G_\infty$.

Recall the random variables $U,U',V,V'$ defined at the beginning of Section~\ref{sec:dynamics}. Now we recall some main results for  the limiting distributions of $M_{U,U'}$ and $M_{V,V'}$  on the random regular graphs $\{G_n\}$. First, every (nonrandom) subsequence $\{G_{n_i}\}$ contains a further subsequence $\{G_{n_{i_j}}\}$ such that the following properties hold $\mathbf P$-a.s.:  $G_{n_{i_j}}$ are connected graphs for all (randomly) large $j$ and 
\begin{align}\label{MUU:RG}
\ms L\left(\frac{M_{U,U'}}{N_{n_{i_j}}}\right)\xrightarrow[j\to\infty]{\rm (d)}\ms L\left(\frac{1}{2}\left(\frac{k-1}{k-2}\right)\mathbf e\right)
\end{align}
\cite[Remark~3.1]{C:MT}. 
Here and in what follows, a meeting time scaled by a constant indexed by $n$ is under $\P^{(n)}$, $\mathbf e$ is exponential with mean $1$, and $\ms L(X)$ denotes the distribution of $X$. Moreover, the convergence \eqref{MUU:RG} extends to the convergence of all moments \cite[Theorem~3.3]{C:MT}. By \eqref{MUU:RG} and \eqref{ergodic},
\begin{align}\label{MUVN}
\ms L\Bigg(\frac{M_{V,V'}}{N_{n_{i_j}}}\Bigg)\xrightarrow[j\to\infty]{\rm (d)} 
\frac{1}{k-1}\delta_0+
\frac{k-2}{k-1}\ms L\left(\frac{1}{2}\left(\frac{k-1}{k-2}\right)\mathbf e\right)\quad \mathbf P\mbox{-a.s.}
\end{align}
See \cite[Section~4]{CCC} for details. We work with $\{G_{n_{i_j}}\}$ and write this subsequence as $\{G_{n}\}$ in the rest of this section.

\begin{prop}\label{prop:phase}
By taking a subsequence of $\{G_n\}$ if necessary,  
\begin{align}\label{MUV}
\ms L\left(\frac{M_{V,V'}}{s_{n}}\right)\xrightarrow[n\to\infty]{\rm (d)} 
\frac{1}{k-1}\delta_0+
\frac{k-2}{k-1}\delta_\infty \quad \mathbf P\mbox{-a.s.}
\end{align}
for any sequence $(s_n)$ such that 
\begin{align}\label{ass:sigman}
\lim_{n\to\infty}s_n=\infty\quad\mbox{and}\quad \lim_{n\to\infty}\frac{s_n}{N_n}=0.
\end{align}
\end{prop}
\begin{proof} 
We fix  two adjacent vertices $a$ and $b$ in $G_\infty$ and give the proof in a few steps. 
 \medskip 

\noindent \hypertarget{Step1}{{\bf Step 1.}}  
We claim that by taking a subsequence of $\{G_n\}$ if necessary, it is possible to choose an auxiliary sequence $(s_n')$ of constants such that
\begin{align}\label{tree}
\lim_{n\to\infty}s_n'=\infty,\quad \lim_{n\to\infty}\frac{s_n'}{N_n}=0,\quad \mbox{and}\quad  \lim_{n\to\infty}\frac{\# \mathcal T_n(s_n')}{N_n}=1\quad\mbox{$\mathbf P$-a.s}.
\end{align}
To find this sequence, first, note that by \eqref{NTn}, we can choose a sequence $(\ell_n)$ such that $\ell_n\to\infty$ and $ \# \mathcal T_n(\ell_n)/N_n\to 1$ in $\mathbf P$-probability. Fix a sequence $(s_n')$ such that $s_n'\leq \ell_n$, $s_n'/N_n\to 0$ and $s_n'\to\infty$. Since  $r\mapsto \# \mathcal T_n(r)$ is decreasing, $ \# \mathcal T_n(s_n')/N_n\to 1$ in $\mathbf P$-probability as well. We have proved the existence of $(s_n')$ satisfying \eqref{tree} such that the third limit holds in the sense of convergence in $\mathbf P$-probability. Hence, by using a subsequence of $\{G_n\}$ if necessary, \eqref{tree} holds. \medskip

\noindent \hypertarget{Step2}{{\bf Step 2.}}  
With respect to the sequence $(s_n')$ chosen in \hyperlink{Step1}{Step 1}, let $(s_n)$ be any slower sequence such that 
\begin{align}\label{sigman}
\lim_{n\to\infty} s_n=\infty\quad\mbox{and}\quad 
\lim_{n\to\infty}\frac{s_n'}{s_n}= \infty.
\end{align} 
By the second limits in \eqref{tree} and \eqref{sigman}, $s_n/N_n\to 0$ so \eqref{ass:sigman} holds. In the next paragraph of this step, we show that 
\begin{align}\label{MUV1}
&\P^{(n)}\left(M_{V,V'}/s_n\in \cdot\,\right)\xrightarrow[n\to\infty]{\rm (d)}\P^{(\infty)}(M_{a,b}<\infty)\delta_0+\P^{(\infty)}(M_{a,b}=\infty)\delta_\infty\quad\mbox{$\mathbf P$-a.s.}
\end{align}
\hyperlink{Step3}{Step 3} will show that the limits in  \eqref{MUV}  and \eqref{MUV1} coincide. Additionally, we will include  in \hyperlink{Step4}{Step 4} the other sequences $(s_n)$ that satisfy \eqref{ass:sigman} but fail to satisfy the second limit in \eqref{sigman}.

Write $J_m$ for the $m$-th jump time of $(B^V,B^{V'})$ on $G_n$. For all  $t\in [0,J_m]$, both $d(V,B^V_t)$ and $ d(V',B^{V'}_t)$ are bounded by $ m$. Hence, on $\{(V,V')\in \mathcal T_n(s_n')\times \mathcal T_n(s_n')\}$, the law of $\{(B^V_t,B^{V'}_t);0\leq t\leq J_{\lfloor s_n'\rfloor/2}\}$ under $\P^{(n)}$ equals the law of $\{(B^a_t,B^b_t);0\leq t\leq J_{\lfloor s_n'\rfloor/2}\}$ under $\P^{(\infty)}$. It follows that 
\begin{align}
&\E^{(n)}\Big[\e^{-\lambda M_{V,V'}/s_n};M_{V,V'}\leq J_{\lfloor s_n'\rfloor/2 }\Big]-\E^{(\infty)}\Big[\e^{-\lambda M_{a,b}/s_n};M_{a,b}\leq J_{\lfloor s_n'\rfloor /2}\Big]\notag\\
=&\sum_{(x,y)\in [\mathcal T_n(s_n')\times \mathcal T_n(s_n')]^\complement}\P\big((V,V')=(x,y)\big)\notag\\
&\quad \times \Big(\E^{(n)}\Big[\e^{-\lambda M_{x,y}/s_n};M_{x,y}\leq J_{\lfloor s_n'\rfloor/2 }\Big]-\E^{(\infty)}\Big[\e^{-\lambda M_{a,b}/s_n};M_{a,b}\leq J_{\lfloor s_n'\rfloor /2}\Big]\Big),\notag
\end{align}
and so
\begin{align}\label{RW11}
\left|\E^{(n)}\big[\e^{-\lambda M_{V,V'}/s_n}\big]-\E^{(\infty)}\big[\e^{-\lambda M_{a,b}/s_n}\big]\right|
\leq 2\frac{N_n-\#\mathcal T_n( s_n')}{N_n}+2\E\big[\e^{-\lambda J_{\lfloor s_n'\rfloor/2}/s_n}\big].
\end{align}
On the right-hand side of \eqref{RW11}, the first term converges to zero $\mathbf P$-a.s. by the third limit in (\ref{tree}), and the choice of $(s_n)$ from (\ref{sigman}) gives $\E[\e^{-\lambda J_{\lfloor s_n'\rfloor/2}/s_n}]=(2/(2+\lambda/s_n))^{\lfloor s_n'\rfloor/2}\to 0$. The right-hand side of \eqref{RW11} thus tends to zero $\mathbf P$-a.s. Additionally,  $\E^{(\infty)}[\e^{-\lambda M_{a,b}/s_n}]\to \P^{(\infty)}(M_{a,b}<\infty)$ by the first limit in \eqref{sigman}. We have proved \eqref{MUV1}.\medskip

\noindent \hypertarget{Step3}{{\bf Step 3.}} In this step,  we show that 
\begin{align}\label{Mxy}
\P^{(\infty)}(M_{a,b}<\infty)=(k-1)^{-1}
\end{align}
and only consider random walks on $G_\infty$.

By symmetry, the hitting time $H_{a,b}$ of $b$ by $B^a$  has the same distribution as $2M_{a,b}$. Hence,   
\begin{align*}
\P^{(\infty)}(M_{a,b}<\infty)&=\P^{(\infty)}(H_{a,b}<\infty)
=\left.\E^{(\infty)}\left[\int_0^\infty \1_{\{b\}}(B^a_t)\d t\right]\right/ \E^{(\infty)}\left[\int_0^\infty\1_{\{b\}}(B^b_t)\d t\right],
\end{align*}
where the second equality follows from a standard Green function decomposition for hitting times of points. The Green functions in the last equality satisfy 
\[
\frac{k-1}{k-2}=
\E^{(\infty)}\left[\int_0^\infty \1_{\{b\}}(B^b_t)\d t\right]=1+\E^{(\infty)}\left[\int_0^\infty \1_{\{b\}}(B^a_t)\d t\right].
\]
Here, the first equality is implied by the Kesten--McKay law for the spectral measure of $G_\infty$ (see \cite[(3.3)]{C:MT} and the references there); the second equality uses the strong Markov property at the first jump time of $B^b$ and the symmetry of $G_\infty$. The identity in \eqref{Mxy} follows from the last two displays. \medskip

\noindent \hypertarget{Step4}{{\bf Step 4.}} To complete the proof, we extend \eqref{MUV1} to all the faster sequences $(s_n)$ such that (\ref{ass:sigman}) holds, but now, $\liminf_{n\to\infty}s_n'/s_n<\infty$. Fix any sequence $(s''_n)$ satisfying $s_n''\to\infty$ and $s_n'/s_n''\to\infty$ as in \hyperlink{Step2}{Step 2}. Then it is enough to show \eqref{MUV1} for all sequences $(s_n)$ satisfying (\ref{ass:sigman}) and $s_n\geq cs_n''$ for some constant $c\in (0,\infty)$.

As recalled above, the convergence in \eqref{MUU:RG} extends to the convergence of all moments. Hence,
\begin{align}\label{gamma0}
\lim_{n\to\infty}\frac{2\E^{(n)}[M_{U,U'}]}{N_n}= \frac{k-1}{k-2} \quad \mathbf P\mbox{-a.s.}
\end{align}
Additionally, by \eqref{MUV1} and \eqref{Mxy}, it holds that
\begin{align}\label{gamma}
\lim_{n\to\infty}\frac{2\E^{(n)}[M_{U,U'}]}{N_n}\P^{(n)}(M_{V,V'}>s_n''t)=1,\quad\forall\;t\in (0,\infty); \quad \mathbf P\mbox{-a.s.}
\end{align}
and so, by \eqref{MUU:RG} and \cite[Proposition~4.3 (2)]{CCC}, \eqref{gamma} with $s_n''$ replaced by $s_n$ holds. We obtain \eqref{ass:sigman} from this limit and \eqref{gamma0}. The proof is complete.
\end{proof}

\begin{rmk}
McKay \cite[Theorem~1.1]{McKay} derives the limiting spectral measures of large random regular graphs. There the randomness of graphs only plays the role of inducing asymptotically deterministic properties. For the present case, we could have worked with given sequences of $k$-regular graphs and obtained the same limit if the graphs have spectral gaps bounded away from zero and are locally tree-like. (Dropping the locally tree-like assumption calls for a different evaluation of the limit.) We choose to work with the above context to explain how the randomness of graphs should be handled for the convergence of the evolutionary game model.  \hfill $\blacksquare$
\end{rmk}

\end{document}